\documentclass{article}
\usepackage{amsmath}
\usepackage[colorlinks = true, citecolor=red, urlcolor=blue,linkcolor=blue]{hyperref}
\usepackage{enumitem}
\usepackage{amsthm}
\usepackage{amssymb}
\usepackage{dsfont}
\usepackage{ulem}
\usepackage{color}
\usepackage{multicol}
\usepackage{mathrsfs}
\usepackage{ngerman}
\usepackage{titling}
\usepackage[lmargin=107.551428pt,rmargin=71.700952pt]{geometry}
\providecommand{\keywords}[1]
{
	\small	
	\textbf{Key words:} #1
}
\providecommand{\ams}[1]
{
	\small	
	\textbf{AMS-Classification:} #1
}
	\title{Strong Nonlocal-to-Local Convergence of the Cahn-Hilliard Equation and its Operator}
	\date{}
	\author{
	Helmut Abels\thanks{Fakultät für Mathematik,
		Universität Regensburg,
		93040 Regensburg,
		Germany, e-mail: \href{helmut.abels@mathematik.uni-regensburg.de}{helmut.abels@mathematik.uni-regensburg.de},}$\;$ and Christoph Hurm
		\thanks{Fakultät für Mathematik,
		Universität Regensburg,
		93040 Regensburg,
		Germany, e-mail: \href{christoph.hurm@mathematik.uni-regensburg.de}{christoph.hurm@mathematik.uni-regensburg.de}}
		}

\begin{document}
	\maketitle
	\renewcommand\abstractname{Abstract}
	\begin{abstract}
		We prove convergence of a sequence of weak solutions of the nonlocal Cahn-Hilliard equation to the strong solution of the corresponding local Cahn-Hilliard equation. The analysis is done in the case of sufficiently smooth bounded domains with Neumann boundary condition and a $W^{1,1}$-kernel. The proof is based on the relative entropy method.
		Additionally, we prove the strong $L^2$-convergence of the nonlocal operator to the negative Laplacian together with a rate of convergence.
	\end{abstract}
\newline
\ams{35B40, 35K25, 35K55, 35D30, 45K05}
	%-----------------------------------
	%	Abstract
	%-----------------------------------
	
	\section{Introduction}\label{intro}
	The Cahn-Hilliard equation was originally introduced in \cite{Cahn} to model the phenomena of spinodal decomposition in binary alloys. Since then, it has been frequently used in a variety of different mathematical models describing phenomena such as population dynamics, image processing, two-phase flows and tumor growth, cf. \cite{Davoli1, Davoli2, Elbar,Knopf, Melchoinna}.
	\newline
	The (local) Cahn-Hilliard equation as introduced in \cite{Cahn} reads as follows:
	\begin{align}
		\partial_tc &= m\Delta\mu \;\;\;\;\;\;\;\;\;\;\;\;\;\;\;\;\;\text{in}\;\Omega_T,\label{CH1} \\
		\mu &= -\Delta c + f^\prime(c) \;\;\;\;\;\;\text{in}\;\Omega_T,\label{CH2}
	\end{align}
	where $\Omega_T:= (0,T)\times\Omega$, $T>0$ and $\Omega\subset\mathbb{R}^n$, $n\leq 3$, is a bounded domain with $C^2$-boundary. Further, we require the following
	initial condition
	\begin{align}\label{CHIC}
		c|_{t=0} = c_0\;\;\;\;\;\;\text{in }\Omega
	\end{align}
	and boundary conditions
	\begin{align}\label{CHBC}
		\partial_{\mathbf{n}}c=0,\;\;\partial_{\mathbf{n}}\mu=0\;\;\;\;\;\;\;\;\;\;\text{on }\partial\Omega\times(0,T).
	\end{align}
	Here, $c$ is a concentration parameter and $\mu$ the chemical potential associated to $c$. Furthermore, $m>0$ is the mobility coefficient and $f$ is the free energy density. Typical choices for the potential in the free energy density are a smooth double-well potential, e.g. $f(c) := K(1-c^2)^2$ for some constant $K>0$, or the logarithmic potential $f(c) := \frac{\theta}{2}\big((1-c)\ln(1-c) + (1+c)\ln(1+c)\big) - \frac{\theta_c}{2}c^2$ for $c\in[-1,1]$, where we assume $0<\theta<\theta_c$. We note that the Cahn-Hilliard equation describes the $H^{-1}$-gradient flow of the free energy functional 
	\begin{align}\label{freeenergyCH}
		\mathcal{E}^{CH}(c) := \int_{\Omega}\frac{1}{2}|\nabla c|^2 + f(c)\:\text{d}x.
	\end{align}
	The Cahn-Hilliard equation has already been studied very intensively and there exists an extensive literature (see \cite{Davoli1, Davoli2, Elbar,Knopf, Melchoinna} and the references therein).
	\newline
	\newline
	The nonlocal counterpart of the Cahn-Hilliard equation has first been presented by Giacomin and Lebowitz in \cite{Giacomin}, where the authors considered the hydrodynamic limit of a microscopic model describing an $n$-dimensional lattice gas evolving via the (Poisson) nearest neighbour exchange process. 
	The nonlocal Cahn-Hilliard equation can be interpreted as the $H^{-1}$-gradient flow of the non-local free energy functional
	\begin{align}\label{nonlocenergy}
	\mathcal{E}^{NL}_\varepsilon(c) := \frac{1}{4}\int_{\Omega}\int_{\Omega}J_\varepsilon(x-y)\big|c(x) - c(y)\big|^2\:\text{d}y\text{d}x + \int_{\Omega}f(c(x))\:\text{d}x.
	\end{align}
	This leads to the following system:
	\begin{align}
		\partial_tc &= m\Delta\mu \;\;\;\;\;\;\;\;\;\;\;\;\;\text{in}\;\Omega_T,\label{NLCH1} \\
		\mu &=  \mathcal{L}_\varepsilon c + f^\prime(c)\;\;\;\;\text{in }\Omega_T,\label{NLCH2}
	\end{align}
	where we define 
	\begin{align*}
		\mathcal{L}_\varepsilon c(x) := -\int_{\Omega}J_\varepsilon(|x-y|)c(y)\:\text{d}y + \int_{\Omega}J_\varepsilon(|x-y|)c(x)\:\text{d}y\;\;\;\text{for all }x\in\Omega.
	\end{align*}
	Further, we require the
	initial condition
	\begin{align}\label{NLCHIC}
	c|_{t=0} = c_{0}\;\;\;\;\;\;\text{in }\Omega
	\end{align}
	and the boundary condition
	\begin{align}\label{NLCHBC}
	\partial_{\mathbf{n}}\mu=0\;\;\;\;\;\;\;\;\;\;\text{on }\partial\Omega\times(0,T).
	\end{align}
	Here, $J_\varepsilon: \mathbb{R}^n\rightarrow[0,\infty)$ is a non-negative function. More precisely, we assume that $J_\varepsilon(x) = \frac{\rho_\varepsilon(x)}{|x|^2}$ for all $x\in\mathbb{R}^n$ and $J_\varepsilon\in W^{1,1}(\mathbb{R}^n)$, where $(\rho_\varepsilon)_{\varepsilon>0}$ is a family of mollifiers satisfying  
	\begin{align*}
		&\rho_\varepsilon: \mathbb{R}\rightarrow [0,\infty),\;\;\;\rho_\varepsilon\in L^1(\mathbb{R}),\;\;\;\rho_\varepsilon(r)=\rho_\varepsilon(-r)\;\;\;\text{for all }r\in\mathbb{R},\;\varepsilon>0, \\
		&\int_{0}^\infty\rho_\varepsilon(r)\:r^{n-1}\:\text{d}r = \frac{2}{C_n}\;\;\;\text{for all }\varepsilon>0, \\
		&\lim\limits_{\varepsilon\searrow 0}\int_{\delta}^\infty\rho_\varepsilon(r)\:r^{n-1}\:\text{d}r = 0\;\;\;\text{for all }\delta>0,
	\end{align*}
	where $C_n := \int_{\mathbb{S}^{n-1}}|e_1\cdot\sigma|^2\:\text{d}\mathcal{H}^{n-1}(\sigma)$.
	Moreover, the singular potential $f$ obeys the same assumptions as in \cite{Gal}, i.e.  
	\begin{align*}
	&f\in C^0([-1,1])\cap C^1(-1,1),\\
	&\lim\limits_{s\rightarrow -1}f^\prime(s) = -\infty,\;\;\;\;\lim\limits_{s\rightarrow 1}f^\prime(s)= +\infty,\;\;\;\;f^{\prime\prime}(s) \geq -\alpha > 0.
	\end{align*}
	Observe that the logarithmic potential mentioned above fulfills these assumptions.
	In our analysis, it is also possible to consider regular potentials $f:\mathbb{R}\rightarrow\mathbb{R}$, which satisfy $f^{\prime\prime}(s) \geq -\alpha$ for all $s\in\mathbb{R}$ together with a growth condition $|f^\prime(s)| \leq C(|s|^3+1)$ for all $s\in\mathbb{R}$, since we only need that $f^{\prime\prime}$ is bounded from below. A typical choice for $f$ is the double-well potential $f(c) := K(1-c^2)^2$ for some constant $K>0$.

	The nonlocal Cahn-Hilliard equation has already been subject to an intense research activity in the recent years. 
	For instance, in the case of singular potentials, the authors in \cite{Gal} proved well-posedness and regularity of weak solutions. Moreover, they established the validity of the strict separation property in two spatial dimensions. 
	For further results on the nonlocal Cahn-Hilliard equation, we refer the reader to \cite{Davoli1, Davoli2, Elbar, Gal,Knopf, Melchoinna} and the references therein. 
	\newline
	\newline
	It is the goal of this contribution to show strong convergence of the nonlocal operator $\mathcal{L}_\varepsilon$ to $-\Delta$ with respect to the $L^2$-topology and to show convergence of weak solutions of \eqref{NLCH1} - \eqref{NLCH2} together with \eqref{NLCHIC} - \eqref{NLCHBC} to the strong solution of \eqref{CH1} - \eqref{CH2} together with \eqref{CHIC} - \eqref{CHBC} under suitable conditions on the initial values.
	
	Nonlocal-to-local asymptotics have already been studied in \cite{Vasquez}, where the author proved the convergence of weak solutions of the fractional heat equation to the fundamental solution as $t\rightarrow\infty$. In \cite{Crismale}, the authors studied the limits $s\rightarrow 0^+$ and $s\rightarrow 1^-$ for s-fractional heat flows in a cylindrical domain with homogeneous Dirichlet boundary conditions. Convergence of nonlocal quadratic forms to local quadratic forms of gradient type was shown by Foghem Gounoue et al. in \cite{Kassmann}.
	 
	Convergence of solutions of the nonlocal Cahn-Hilliard equation, i.e., \eqref{NLCH1} - \eqref{NLCH2}, to the local Cahn-Hilliard equation, i.e., \eqref{CH1} - \eqref{CH2}, has already been proved by Melchionna et al. in \cite{Melchoinna} in the case of periodic boundary conditions and  a regular free energy density and by Davoli et al. in \cite{Davoli1} in the case of periodic boundary conditions and a singular free energy density. In the case of Neumann boundary conditions, convergence has been proved by Davoli et al. in \cite{Davoli3} with an additional viscosity term in the nonlocal Cahn-Hilliard equation and in \cite{Davoli2} for $W^{1,1}$-kernels. The authors in \cite{Elbar} proved convergence of the nonlocal to the local degenerate Cahn-Hilliard equation. 
	In \cite{AbelsTerasawa1} and \cite{Liang}, the authors proved the nonlocal-to-local limit for a coupled Navier--Stokes/Cahn--Hilliard system.
	We also mention that these results are based on the works by J. Bourgain, H. Brezis and P. Mironescu \cite{Bourgain}, where the authors proved nonlocal-to-local convergence for functionals of the form \eqref{nonlocenergy}, as well as the $\Gamma$-convergence results by A.C. Ponce in \cite{Ponce1, Ponce2}. 
	 
	 In this contribution, however, we use Fourier transforms, which then guarantee a rate of convergence for the nonlocal operator in the case of $\Omega = \mathbb{R}^n$. Using a reflection argument together with perturbation and a localization argument, we can even show a rate of convergence in sufficiently smooth bounded domains $\Omega\subset\mathbb{R}^n$.
	\newline
	\newline
	The structure of this paper is the following: In Section \ref{preliminaries}, we recall some definitions and preliminary results. In Section \ref{convop}, we prove first convergence results of the nonlocal operator $\mathcal{L}_\varepsilon$. In Section \ref{main}, we then state and prove the main theorem about the strong convergence of the nonlocal operator. This will be done by localization. Finally, in Section \ref{nonloctoloc}, we apply our results from Section \ref{main} to prove nonlocal-to-local convergence of the Cahn-Hilliard equation using the relative entropy method.

	\section{Preliminaries}\label{preliminaries}
	%-----------------------------------
	%	Notation
	%-----------------------------------
	In this section, we recollect some preliminary results, which we need throughout the paper. First, we briefly recall the Fourier transform $\mathcal{F}: L^2(\mathbb{R}^n) \rightarrow L^2(\mathbb{R}^n)$ given by
	\begin{align*}
		\mathcal{F}(f)(\xi) := \int_{\mathbb{R}^n}e^{-ix\cdot\xi}f(x)\:\text{d}x. 
	\end{align*}
	We observe that this map is well-defined and defines an isometric automorphism on $L^2(\mathbb{R}^n)$, cf. Plancherel`s Theorem. 
	Further, we recall that the inverse of the negative Laplacian $-\Delta$ with Neumann boundary condition is a well-defined isomorphism
	\begin{align*}
		(-\Delta)^{-1} : \{c\in (H^1(\Omega))^\prime: c_\Omega=0\} \rightarrow \{c\in H^1(\Omega): c_\Omega=0\}. 
	\end{align*}
	For $c\in (H^1(\Omega))^\prime$, we define $c_\Omega = \frac{1}{|\Omega|}\langle c,1\rangle$.
	Next, we state some important inequalities.
	\newtheorem{lemma0}{Lemma}[section]
	\begin{lemma0}\label{lemma0}
		For every $\delta>0$ there exist constants $C_\delta>0$ and $\varepsilon_\delta>0$ with the following properties:
		\begin{enumerate}
			\item  For every sequence $(f_\varepsilon)_{\varepsilon>0}\subset H^1(\Omega)$, there holds
			\begin{align}
				\|f_{\varepsilon_1}-f_{\varepsilon_2}\|_{H^1(\Omega)}^2 &\leq \delta\int_{\Omega}\int_{\Omega}J_{\varepsilon_1}(x,y)\big|\nabla f_{\varepsilon_1}(x) - \nabla f_{\varepsilon_2}(y)\big|^2\:\textup{d}y\textup{d}x \nonumber\\
				&+ \delta \int_{\Omega}\int_{\Omega}J_{\varepsilon_2}(x,y)\big|\nabla f_{\varepsilon_1}(x) - \nabla f_{\varepsilon_2}(y)\big|^2\:\textup{d}y\textup{d}x + C_\delta\|f_{\varepsilon_1} - f_{\varepsilon_2}\|_{L^2(\Omega)}^2.
			\end{align}
			\item For every sequence $(f_\varepsilon)_{\varepsilon>0}\subset L^2(\Omega)$, there holds
			\begin{align}\label{ineqDav}
				\|f_{\varepsilon_1} - f_{\varepsilon_2}\|_{L^2(\Omega)}^2 \leq \delta\mathcal{E}_{\varepsilon_1}(f_{\varepsilon_1}) + \delta\mathcal{E}_{\varepsilon_2}(f_{\varepsilon_2}) + C_\delta\|f_{\varepsilon_1} - f_{\varepsilon_2}\|_{(H^1(\Omega))^*}^2.
			\end{align}
		\end{enumerate}
	\end{lemma0}
	\begin{proof}
		For a proof, we refer to \cite[Lemma 4(2)]{Davoli1}.
	\end{proof}
	Here, the term $\mathcal{E}_\varepsilon$ is defined by
	\begin{align*}
		\mathcal{E}_\varepsilon(c) := \frac{1}{4}\int_{\Omega}\int_{\Omega}J_\varepsilon(x-y)\big|c(x) - c(y)\big|^2\:\text{d}y\text{d}x
	\end{align*}
	for all $c\in H^1(\Omega)$, i.e., the first part of the nonlocal energy functional $\mathcal{E}^{NL}_\varepsilon$ in \eqref{nonlocenergy}. In the limit $\varepsilon\searrow 0$, this term behaves as 
	\begin{align}\label{nonlockonv}
		\lim\limits_{\varepsilon\searrow 0}\mathcal{E}_\varepsilon(c) = \frac{1}{2}\int_{\Omega}\big|\nabla c(x)\big|^2\:\text{d}x
	\end{align}
	for all $c\in H^1(\Omega)$. For details, we refer to \cite{Bourgain, Davoli1}.
	\newtheorem{lemTraf}[lemma0]{Lemma}
	\begin{lemTraf}\label{lemTraf}
		Let $a,b\in L^1(\mathbb{R}^{n-1})$, $n\geq 2$. Then, it holds
		\begin{align*}
			\int_{\mathbb{R}^{n-1}}\int_{\mathbb{R}^{n-1}}a(x-y)\:b((1-t)y+tx)\:\textup{d}x\textup{d}y = \|a\|_{L^1(\mathbb{R}^{n-1})}\|b\|_{L^1(\mathbb{R}^{n-1})}
		\end{align*}
		for all $t\in[0,1]$.
	\end{lemTraf}
	\begin{proof}
		We define the mapping 
		\begin{align*}
			\Phi_t: \mathbb{R}^{n-1}\times\mathbb{R}^{n-1} &\rightarrow \mathbb{R}^{n-1}\times\mathbb{R}^{n-1}, \\
			(x,y) &\mapsto (x-y, (1-t)y +tx).
		\end{align*}
		Then, we have
		\begin{align*}
			\det\text{D}\Phi_t(x,y) = \det\left(\begin{array}{cc}
			\text{Id} & -\text{Id} \\
			t\text{Id} & (1-t)\text{Id}
			\end{array}\right) = 1
		\end{align*}
		for all $t\in[0,1]$. Using change of variables and Fubini`s Theorem, we obtain
		\begin{align*}
			\int_{\mathbb{R}^{n-1}}\int_{\mathbb{R}^{n-1}}a(x-y)\:b((1-t)y+tx)\:\textup{d}x\textup{d}y &= \left(\int_{\mathbb{R}^{n-1}}a(x)\:\text{d}x\right)\left(\int_{\mathbb{R}^{n-1}}b(y)\:\text{d}y\right) \\
			&=\|a\|_{L^1(\mathbb{R}^{n-1})}\|b\|_{L^1(\mathbb{R}^{n-1})}
		\end{align*}
		for all $t\in[0,1]$.
	\end{proof}
	\newtheorem{remdense}[lemma0]{Remark}
	\begin{remdense}\label{remdense}
		\upshape
		The space $\{c\in C^k_c(\overline{\mathbb{R}^n_+}): \partial_{\mathbf{n}}c|_{\partial\mathbb{R}^n_+}=0\}$ is dense in
		$\{c\in H^k(\mathbb{R}^n_+): \partial_{\mathbf{n}}c|_{\partial\mathbb{R}^n_+}=0\}$. If $k=3$, the proof is based on the following idea:
		\newline
		Since $C^\infty_0(\overline{\mathbb{R}^n_+})$ is dense in $H^3(\mathbb{R}^n_+)$, for any $c\in H^3(\mathbb{R}^n_+)$ there exists a sequence $(\tilde{c}_j)_{j\in\mathbb{N}} \subset C^\infty_0(\overline{\mathbb{R}^n_+})$ such that $\tilde{c}_j \rightarrow c$ in $H^3(\mathbb{R}^n_+)$ as $j\rightarrow\infty$. Next, we consider the auxiliary problem
		\begin{align*}
			(1-\Delta)w_j &= 0\;\;\;\;\;\;\;\;\;\;\;\;\:\text{in }\mathbb{R}^n_+, \\
			\mathbf{n}\cdot\nabla w_j &= \mathbf{n}\cdot\nabla \tilde{c}_j\;\;\;\text{on }\partial\mathbb{R}^n_+
		\end{align*} 
		for all $j\in\mathbb{N}$. Then, by linear elliptic theory, there exists a solution $w_j\in C^3(\overline{\mathbb{R}^n_+})$ for all $j\in\mathbb{N}$. Finally, the sequence $c_j := \tilde{c}_j - w_j$, $j\in\mathbb{N}$, has the desired properties.
		\newline
		\newline
		Next, we consider the case of the bent half-space.
		Let $\gamma\in C^{k}_b(\mathbb{R}^{n-1})$ be given. Then, the space $\{c\in C^{k-1,1}_c(\overline{\mathbb{R}^n_\gamma}):\partial_{\mathbf{n}}c|_{\partial\mathbb{R}^n_\gamma}=0\}$ is a dense subset of $\{c\in H^k(\mathbb{R}^n_\gamma): \partial_{\mathbf{n}}c|_{\partial\mathbb{R}^n_\gamma}=0\}$. The proof is based on the following idea:
		\newline
		Since $\gamma\in C^{k}_b(\mathbb{R}^{n-1})$, there exists a $C^{k-1,1}$-diffeomorphism $F_\gamma : \mathbb{R}^n \rightarrow \mathbb{R}^n$ such that $F_\gamma(\mathbb{R}^n_+) = \mathbb{R}^n_\gamma$, $F_\gamma(x^\prime,0) = (x^\prime,\gamma(x^\prime))$ and $-\partial_{x_n}F_\gamma(x)|_{x_n=0} = \mathbf{n}(x^\prime,\gamma(x^\prime))$, where $\mathbf{n}$ denotes the exterior unit normal on $\partial\mathbb{R}^n_\gamma$, cf. \cite[Lemma 2.1]{Schumacher}. Let $c\in \{c\in H^k(\mathbb{R}^n_\gamma): \partial_{\mathbf{n}}c|_{\partial\mathbb{R}^n_\gamma}=0\}$. We define $\tilde{c} := c\circ F_\gamma \in H^k(\mathbb{R}^n_+)$. Thanks to the first part, there exists a sequence $(\tilde{c}_j)_{j\in\mathbb{N}}\subset \{c\in C^k_c(\overline{\mathbb{R}^n_+}): \partial_{\mathbf{n}}c|_{\partial\mathbb{R}^n_+}=0\}$ such that $\tilde{c}_j \rightarrow \tilde{c}$ in $H^k(\mathbb{R}^n_+)$ as $j\rightarrow\infty$. Finally, the sequence $c_j := \tilde{c}_j\circ F_\gamma^{-1} \in C^{k-1,1}_c(\overline{\mathbb{R}^n_\gamma})$, $j\in\mathbb{N}$, has the desired properties.
	\end{remdense}
	
	\section{Convergence of the Nonlocal to the Local Operator}\label{convop}
	In this section, we prove the strong $L^2-$convergence of the nonlocal operator $\mathcal{L}_\varepsilon$ to $-\Delta$. In the first case, we study the convergence on $\mathbb{R}^n$.
	\newtheorem{lemma1}[lemma0]{Lemma}
	\begin{lemma1}\label{lemma1}
		Let $c\in H^2(\mathbb{R}^n)$. Then, it holds
		\begin{align*}
			\Big\|\mathcal{L}_\varepsilon^{\mathbb{R}^n} c + \Delta c\Big\|_{L^2(\mathbb{R}^n)} \rightarrow 0\;\;\;\text{as }\varepsilon\searrow 0.
		\end{align*}
		In addition, if $c\in H^3(\mathbb{R}^n)$, we even have
		\begin{align*}
		\Big\|\mathcal{L}^{\mathbb{R}^n}_\varepsilon c + \Delta c\Big\|_{L^2(\mathbb{R}^n)} 
		\leq K\varepsilon\|c\|_{H^3(\mathbb{R}^n)}.
		\end{align*}
	\end{lemma1}
	\begin{proof}
		Thanks to Plancherel`s Theorem, it suffices to prove 
		\begin{align*}
			\Big\|\widehat{\mathcal{L}^{\mathbb{R}^n}_\varepsilon c} + \widehat{\Delta c}\Big\|_{L^2(\mathbb{R}^n)}\rightarrow 0\;\;\;\text{as }\varepsilon\searrow 0.
		\end{align*}
		By definition, we have
		\begin{align*}
			\Big\|\widehat{\mathcal{L}^{\mathbb{R}^n}_\varepsilon c} + \widehat{\Delta c}\Big\|_{L^2(\mathbb{R}^n)}^2 = \frac{1}{(2\pi)^n} \int_{\mathbb{R}^n}\big|\big(-\mathcal{F}(J_\varepsilon)(\xi) + \mathcal{F}(J_\varepsilon)(0) - |\xi|^2\big)\mathcal{F}(c)(\xi)\big|^2\text{d}\xi,
		\end{align*}
		where we used that 
		\begin{align*}
		(J_\varepsilon*1)(x) = \mathcal{F}(J_\varepsilon)(0)
		\end{align*}
		for all $x\in\mathbb{R}^n$. Next, we prove the pointwise convergence 
		\begin{align*}
		\mathcal{F}(J_\varepsilon)(0) - \mathcal{F}(J_\varepsilon)(\xi) \rightarrow |\xi|^2\;\;\;\text{for all }\xi\in\mathbb{R}^n
		\end{align*}
		as $\varepsilon\searrow0$. Defining the functions $f_\xi(x) := -e^{-ix\cdot\xi}$, we have
		\begin{align*}
		\mathcal{F}(J_\varepsilon)(0) - \mathcal{F}(J_\varepsilon)(\xi) &= \int_{\mathbb{R}^n}J_\varepsilon(|x|)(f_\xi(x)-f_\xi(0))\:\text{d}x \\
		&= \int_{\mathbb{R}^n}J_\varepsilon(|x|)\big(f_\xi(x)-f_\xi(0) - \frac{1}{2}x^TD^2f_\xi(0)x\big)\:\text{d}x \\
		&\;\;\;+ \int_{\mathbb{R}^n}J_\varepsilon(|x|)\frac{1}{2}x^TD^2f_\xi(0)x\:\text{d}x \\
		&=: I_\varepsilon^1 + I_\varepsilon^2.
		\end{align*}
		Now, we analyze these terms separately.
		
		Ad $I_\varepsilon^1$: Since the function $J_\varepsilon(|x|)x_i$ is odd for all $i=1,\ldots,n$, it holds
		\begin{align}\label{equ1}
		\int_{\mathbb{R}^n}J_\varepsilon(|x|)x_i\:\text{d}x = 0
		\end{align}
		for all $i=1,\ldots,n$. Therefore, multiplying \eqref{equ1} by $\partial_if_\xi(0)$ and summing over all $i=1,\ldots,n$, gives 
		\begin{align*}
		\big|I_\varepsilon^1\big| = \left|\int_{\mathbb{R}^n}J_\varepsilon(|x|)\big(f_\xi(x) - f_\xi(0) - \nabla f_\xi(0)\cdot x - \frac{1}{2}x^TD^2f_\xi(0)x\big)\:\text{d}x\right|.
		\end{align*}
		Using Taylor`s Theorem, we have $|f_\xi(x) - f_\xi(0) - \nabla f_\xi(0)\cdot x - \frac{1}{2}x^TD^2f_\xi(0)x| \leq \tilde{\varepsilon}C_{d,f}|x|^2$ for all $|x|<\delta$. Therefore, it holds
		\begin{align*}
		\big|I_\varepsilon^1\big| &\leq \int_{\mathbb{R}^n}J_\varepsilon(|x|)\Big|f_\xi(x) - f_\xi(0) - \nabla f_\xi(0)\cdot x - \frac{1}{2}x^TD^2f_\xi(0)x\Big|\:\text{d}x \\
		&\leq \int_{B_\delta(0)}\tilde{\varepsilon}C_{d,f}\rho_\varepsilon(|x|)\:\text{d}x + \int_{B_\delta(0)^c}C\rho_\varepsilon(|x|)\:\text{d}x.
		\end{align*}
		Since we can choose $\tilde{\varepsilon}>0$ arbitrarily small and since $\int_{B_\delta(0)}\rho_\varepsilon(|x|)\:\text{d}x\leq K$, the first term in the last line is arbitrarily small. In the second integral, the properties of $\rho_\varepsilon$ imply that this integral vanishes as $\varepsilon\searrow0$. Altogether, this shows $I_\varepsilon^1\rightarrow0$ as $\varepsilon\searrow0$.
		
		Ad $I_\varepsilon^2$: We compute
		\begin{align*}
		I_\varepsilon^2 &= \int_{\mathbb{R}^n}J_\varepsilon(|x|)\frac{1}{2}\sum_{l,m=1}^nx_l\partial_l\partial_mf_\xi(0)x_m\:\text{d}x\\
		&= \frac{1}{2}\sum_{l,m=1}^n\xi_l\xi_m\int_{\mathbb{R}^n}J_\varepsilon(|x|)x_lx_m\:\text{d}x \\
		&= \frac{1}{2}\sum_{m=1}^n\xi_m^2\int_{\mathbb{R}^n}J_\varepsilon(|x|)x_m^2\:\text{d}x \\
		&= \frac{1}{2}\sum_{m=1}^n\xi_m^2\frac{1}{n}\int_{\mathbb{R}^n}J_\varepsilon(|x|)\sum_{j=1}^nx_j^2\:\text{d}x
		= |\xi|^2.
		\end{align*}
		Here, we used the facts
		\begin{align*}
		\int_{\mathbb{R}^n}J_\varepsilon(|x|)x_lx_m\:\text{d}x = 0
		\end{align*}
		for all $m\neq l$, and 
		\begin{align*}
		\int_{\mathbb{R}^n}J_\varepsilon(|x|)x_m^2\:\text{d}x = \int_{\mathbb{R}^n}J_\varepsilon(|x|)x_1^2\:\text{d}x
		\end{align*}
		for all $m=1,\ldots,n$. In the last step, we used our assumptions on $\rho_\varepsilon$ to compute
		\begin{align*}
			\frac{1}{2n}\int_{\mathbb{R}^n}\rho_\varepsilon(|x|)\:\text{d}x = \frac{1}{2n}\:\omega_n\int_{0}^\infty\rho(r)r^{n-1}\:\text{d}r = \frac{\omega_n}{n}\frac{1}{C_n} = 1,
		\end{align*}
		where we used $\omega_n = \mathcal{H}^{n-1}(\mathbb{S}^{n-1})$ and that (for any $j=1,\ldots,n$) it holds
		\begin{align*}
			C_n = \int_{\mathbb{S}^{n-1}}|e_1\cdot\sigma|^2\:\text{d}\mathcal{H}^{n-1}(\sigma) = \int_{\mathbb{S}^{n-1}}|e_j\cdot\sigma|^2\:\text{d}\mathcal{H}^{n-1}(\sigma) = \frac{1}{n}\int_{\mathbb{S}^{n-1}}|\sigma|^2\:\text{d}\mathcal{H}^{n-1}(\sigma) = \frac{\omega_n}{n}.
		\end{align*}
		Altogether, this shows the pointwise convergence
		\begin{align*}
		\mathcal{F}(J_\varepsilon)(\xi) - \mathcal{F}(J_\varepsilon)(0) \rightarrow |\xi|^2\;\;\;\text{as }\varepsilon\searrow 0
		\end{align*}
		for all $\xi\in\mathbb{R}^n$. Again, using Taylor`s Theorem, we observe that 
		\begin{align}
		\Big|\widehat{J_\varepsilon}(0) - \widehat{J_\varepsilon}(\xi) - |\xi|^2\Big| &= \Big|\int_{\mathbb{R}^n}J_\varepsilon(|x|)\big(f_\xi(x) - f_\xi(0) - \nabla f_\xi(0)\cdot x - \frac{1}{2}x^TD^2f_\xi(0)x\big)\:\text{d}x\Big| \nonumber\\
		&\leq C(1+|\xi|^2)\int_{\mathbb{R}^n}\rho_\varepsilon(|x|)\:\text{d}x \leq C(1+|\xi|^2), \label{equ2}
		\end{align}
		which implies
		\begin{align*}
		\big|\big(\mathcal{F}(J_\varepsilon)(0) - \mathcal{F}(J_\varepsilon)(\xi)- |\xi|^2\big)\mathcal{F}(c)(\xi)\big|^2 \leq C\big(|\mathcal{F}(c)(\xi)|^2 + |\xi|^2|\mathcal{F}(c)(\xi)|^2\big)
		\end{align*}
		for all $\xi\in\mathbb{R}^n$. As $c\in H^2(\mathbb{R}^n)$, the right-hand side is integrable. Therefore, we use Lebesgue`s dominated convergence theorem to conclude the proof for $c\in H^2(\mathbb{R}^n)$.
		
		Now, let $c\in H^3(\mathbb{R}^n)$. Then, we even have
		\begin{align*}
		\Big|\widehat{J_\varepsilon}(0) - \widehat{J_\varepsilon}(\xi) - |\xi|^2\Big| &\leq C\int_{\mathbb{R}^n}J_\varepsilon(|x|)\sup_{y\in\mathbb{R}^n}|D^3f_\xi(y)||x|^3\:\text{d}x \\
		&\leq C|\xi|^3\int_{\mathbb{R}^n}\rho_\varepsilon(|x|)|x|\:\text{d}x \leq C\varepsilon|\xi|^3
		\end{align*}
		for all $\xi\in\mathbb{R}^n$ using a third-order Taylor expansion. Consequently, it holds
		\begin{align*}
		\Big\|\mathcal{L}^{\mathbb{R}^n}_\varepsilon c + \Delta c\Big\|_{L^2(\mathbb{R}^2)}^2 &= \frac{1}{(2\pi)^n}\int_{\mathbb{R}^n}|(\mathcal{F}(J_\varepsilon)(\xi) - \mathcal{F}(J_\varepsilon)(0) - |\xi|^2|)\mathcal{F}(c)(\xi)|^2\:\text{d}\mathbf{\xi} \\
		&\leq \frac{C\varepsilon^2}{(2\pi)^n}\int_{\mathbb{R}^n}\big||\xi|^3\mathcal{F}(c)(\xi)\big|^2\text{d}\mathbf{\xi} = C\varepsilon^2\|c\|_{H^3(\mathbb{R}^n)}^2.
		\end{align*}
	 	Therefore, we obtain
		\begin{align*}
		\Big\|\mathcal{L}^{\mathbb{R}^n}_\varepsilon c + \Delta c\Big\|_{L^2(\mathbb{R}^n)} \leq C\varepsilon\|c\|_{H^3(\mathbb{R}^n)},
		\end{align*}
		which concludes the proof.
	\end{proof}
	In the next lemma, we study the situation, where $x$ and $y$ have positive distance, i.e., no singularity appears. In fact, this will play an important role for the following proofs. 
	\newpage
	\newtheorem{lemma2}[lemma0]{Lemma}
	\begin{lemma2}\label{lemma2}
		Let $\Omega\subset\mathbb{R}^n$ be open and let $\Omega^\prime\subseteq\Omega$ such that $\textup{dist}(\overline{\Omega^\prime},\partial\Omega)>0$. Then, for every $c\in L^2(\Omega)$, it holds
		\begin{align*}
			\big\|\mathcal{R}_\varepsilon c\big\|_{L^2(\Omega^\prime)} \leq K\:\varepsilon\|c\|_{L^2(\Omega)}
		\end{align*}
		where $K>0$ only depends on $\textup{dist}(\Omega^c, \Omega^\prime)$. Here, we defined 
		\begin{align*}
		\mathcal{R}_\varepsilon c(x) := \int_{\Omega^c}J_\varepsilon(|x-y|)\big(c(x) - \tilde{c}(y)\big)\:\text{d}y
		\end{align*}
		for all $c\in L^2(\Omega)$ and almost all $x\in\Omega$. Here, $\tilde{c}\in L^2(\mathbb{R}^n)$ is an extension of $c$ to the whole of $\mathbb{R}^n$ such that $\|\tilde{c}\|_{L^2(\mathbb{R}^n)} \leq K\|c\|_{L^2(\Omega)}$.
	\end{lemma2}
	\begin{proof}
		First of all, we observe
		\begin{align*}
			\big\|\mathcal{R}_\varepsilon c\big\|_{L^2(\Omega^\prime)}^2 
			%&= \int_{\Omega^\prime}\left|\int_{\Omega^c}J_\varepsilon(|x-y|)\big(c(x) - \tilde{c}(y)\big)\:\text{d}y\right|^2\:\text{d}x \\
			&\leq K\int_{\Omega^\prime}\left|\int_{\Omega^c}J_\varepsilon(|x-y|)c(x)\:\text{d}y\right|^2\:\text{d}x + K\int_{\Omega^\prime}\left|\int_{\Omega^c}J_\varepsilon(|x-y|)\tilde{c}(y)\:\text{d}y\right|^2\:\text{d}x \\
			&=: K \big(I_\varepsilon^1 +  I_\varepsilon^2\big).
		\end{align*}
		Now, we estimate these terms separately.
		\newline
		Ad $I_\varepsilon^1$: Since $\textup{dist}(\overline{\Omega^\prime},\partial\Omega)>0$, it holds $|x-y|\geq \delta := \text{dist}(\Omega^c, \Omega^\prime)>0$ for all $x\in \Omega^c$, $y\in\Omega^\prime$. This yields 
		\begin{align*}
			I_\varepsilon^1 \leq K_\delta\int_{\Omega^\prime}|c(x)|^2\left(\int_{\Omega^c}\rho_\varepsilon(|x-y|)|x-y|\:\text{d}y\right)^2\:\text{d}x 
			\leq K_\delta\:\varepsilon^2\|c\|_{L^2(\Omega)}^2.
		\end{align*}
		Ad $I_\varepsilon^2$: Here, we estimate 
		\begin{align*}
			I_\varepsilon^2 \leq \int_{\Omega^\prime}\left(\int_{\Omega^c}J_\varepsilon(|x-y|)|\tilde{c}(y)|\:\text{d}y\right)^2\:\text{d}x &\leq \int_{\Omega^\prime}\left(\int_{\Omega^c}J_\varepsilon(|x-y|)\:\text{d}y\right)\left(\int_{\Omega^c}J_\varepsilon(|x-y|)|\tilde{c}(y)|^2\:\text{d}y\right)\:\text{d}x \\		
			&\leq K_\delta\:\varepsilon \int_{\Omega^c}|\tilde{c}(y)|^2\left(\int_{\Omega^\prime}\rho_\varepsilon(|x-y|)|x-y|\:\text{d}x\right)\:\text{d}y \\
			&\leq K_\delta\:\varepsilon^2 \|c\|_{L^2(\Omega)}^2,
		\end{align*}
		where we used $|x-y|\geq \delta := \text{dist}(\Omega^c, \Omega^\prime)>0$ and Fubini`s Theorem. Altogether, we obtain
		\begin{align*}
			\big\|\mathcal{R}_\varepsilon c\big\|_{L^2(\Omega^\prime)} \leq K_\delta\:\varepsilon\|c\|_{L^2(\Omega)}.
		\end{align*}
		This concludes the proof.
	\end{proof}
In the proof of Theorem \ref{theorem1}, we want to use a localization argument. To this end, we now consider the upper half-space $\mathbb{R}^n_+$.
\newtheorem{lemma3}[lemma0]{Lemma}
	\begin{lemma3}\label{lemma3}
		Let $c\in H^3(\mathbb{R}^n_+)$ with $\partial_\mathbf{n}c = 0$ on $\partial\mathbb{R}^n_+$. Then, it holds
		\begin{align*}
		\Big\|\mathcal{L}_\varepsilon^{\mathbb{R}^n_+} c + \Delta c\Big\|_{L^2(\mathbb{R}^n_+)} \leq K\sqrt{\varepsilon}\|c\|_{H^3(\mathbb{R}^n_+)}.
		\end{align*}
	\end{lemma3}
	\begin{proof}
		Denote $\tilde{c}\in H^3(\mathbb{R}^n)$ an extension of $c$ to the whole of $\mathbb{R}^n$ such that $\|\tilde{c}\|_{H^3(\mathbb{R}^n)}\leq K\|c\|_{H^3(\mathbb{R}^n_+)}$. Then, we observe
		\begin{align*}
		\Big\|\mathcal{L}^{\mathbb{R}^n_+}_\varepsilon c + \Delta c\Big\|_{L^2(\mathbb{R}^n_+)}
		&\leq \Big\|\mathcal{L}^{\mathbb{R}^n}_\varepsilon\tilde{c} + \Delta\tilde{c}\Big\|_{L^2(\mathbb{R}^n)} + \|\mathcal{R}_\varepsilon\tilde{c}\|_{L^2(\mathbb{R}^n_+)} \\ 
		&\leq K\varepsilon\|c\|_{H^3(\mathbb{R}^n_+)} + \|\mathcal{R}_\varepsilon\tilde{c}\|_{L^2(\mathbb{R}^n_+)},
		\end{align*}
		where the error term $\mathcal{R}_\varepsilon$ is given by
		\begin{align}\label{errorterm}
			\mathcal{R}_\varepsilon c(x) := \int_{\mathbb{R}^n_-}J_\varepsilon(|x-y|)(c(x)-\tilde{c}(y))\:\text{d}y
		\end{align}
		for a.e. $x\in\mathbb{R}^n_+$. 
	 	We want to prove that $\|\mathcal{R}_\varepsilon\tilde{c}\|_{L^2(\mathbb{R}^n_+)}\rightarrow 0$ as $\varepsilon\searrow 0$. To this end, we first prove the statement for functions $c\in C^3_0(\overline{\mathbb{R}^n_+})$ with $\partial_\mathbf{n}c = 0$ on $\partial\mathbb{R}^n_+$ and use a density argument afterwards to conclude the proof. 
		
		Using the transformation $y=(y_1,\ldots,y_{n-1},y_n) \mapsto (y_1,\ldots,y_{n-1},-y_n) =: \hat{y}$, we have
		\begin{align*}
		\mathcal{R}_\varepsilon c(x) = \int_{\mathbb{R}^n_-}J_\varepsilon(|x-y|)(c(x) - \tilde{c}(y))\text{d}y = \int_{\mathbb{R}^n_+}J_\varepsilon(|x-\hat{y}|)(c(x)- c(\hat{y}))\text{d}y
		\end{align*}
		for a.e. $x\in \mathbb{R}^n_+$. For $\delta>0$, we obtain
		\begin{align}\label{error}
		|\mathcal{R}_\varepsilon c(x)| &\leq \Big|\int_{B_{\delta}(x^\prime)\times(0,\delta)}J_\varepsilon(|x-\hat{y}|)(c(x)- c(\hat{y}))\text{d}y\Big| + \Big|\int_{\mathbb{R}^n_+\setminus \big(B_{\delta}(x^\prime)\times(0,\delta)\big)}J_\varepsilon(|x-\hat{y}|)(c(x)- c(\hat{y}))\text{d}y\Big|
		\end{align}
		for a.e. $x\in \mathbb{R}^n_+$. In the second integral on the right-hand side, we use the calculations in the proof of Lemma \ref{lemma2} to obtain
		\begin{align*}
			\int_{\mathbb{R}^n_+}\Big|\int_{\mathbb{R}^n_+\setminus
				\big(B_{\delta}(x^\prime)\times(0,\delta)\big)}J_\varepsilon(|x-\hat{y}|)(c(x)- c(\hat{y}))\text{d}y\Big|^2\text{d}x \leq K_\delta\:\varepsilon^2\|c\|_{H^3(\mathbb{R}^n_+)}^2, 
		\end{align*}
		since it holds $\text{dist}(x,\hat{y})\geq \delta>0$ for all $\hat{y}\in\mathbb{R}^n_+\setminus
		\big(B_{\delta}(x^\prime)\times(0,\delta)\big)$. Thus, it suffices to consider the first integral on the right-hand side. Here, a first order Taylor expansion yields
		\begin{align*}
		c(x) - c(\hat{y}) &= \nabla c(x)\cdot(x-\hat{y}) + R_2(x,\hat{y}) \\
		&= \nabla_{x^\prime}c(x)\cdot(x^\prime - y^\prime) + \partial_{x_n}c(x)(x_n+y_n)  + R_2(x,\hat{y}),
		\end{align*}
		where $x:=(x^\prime,x_n)\in \mathbb{R}^{n-1}\times\mathbb{R}_+$ and  
		\begin{align*}
			R_2(x,\hat{y}) := \sum_{|\beta|=2}\frac{2}{\beta!}\left(\int_{0}^1(1-t)\text{D}^\beta c(\hat{y}+t(x-\hat{y}))\text{d}t\right)(x-\hat{y})^\beta.
		\end{align*}
		Inserting this, we end up with
		\begin{align}
			\Big|\int_{B_{\delta}(x^\prime)\times(0,\delta)}J_\varepsilon(|x-\hat{y}|)(c(x)- c(\hat{y}))\text{d}y\Big| &\leq \Big|\int_{B_{\delta}(x^\prime)\times(0,\delta)}J_\varepsilon(|x-\hat{y}|)\nabla_{x^\prime}c(x)\cdot(x^\prime - y^\prime)\;\text{d}y\Big| \nonumber\\
			&\;\;\;+\Big|\int_{B_{\delta}(x^\prime)\times(0,\delta)}J_\varepsilon(|x-\hat{y}|)\partial_{x_n}c(x)(x_n+y_n)\;\text{d}y\Big| \nonumber\\
			&\;\;\;+\Big|\int_{B_{\delta}(x^\prime)\times(0,\delta)}J_\varepsilon(|x-\hat{y}|)R_2(x,\hat{y})\;\text{d}y\Big| \nonumber\\
			&=: I_{\varepsilon,\delta}^1 + I_{\varepsilon,\delta}^2 + I_{\varepsilon,\delta}^3. \label{int1-3}
		\end{align}
		Now, we estimate these integrals separately. 
		
		Ad $I_{\varepsilon,\delta}^1$: We observe that the integrand $J_\varepsilon(|x-\hat{y}|)\nabla_{x^\prime}c(x)\cdot(x^\prime - y^\prime)$ is odd with respect to $x^\prime - y^\prime$. Therefore, it holds 
		\begin{align*}
		\int_{B_{\delta}(x^\prime)}J_\varepsilon(|x-\hat{y}|)\nabla_{x^\prime}c(x)\cdot(x^\prime - y^\prime)\text{d}y^\prime = 0,
		\end{align*}
		which then implies $I_{\varepsilon,\delta}^1 = 0$.
		
		Ad $I_{\varepsilon,\delta}^2$: First of all, the properties of $c$ and the fundamental theorem of calculus imply
		\begin{align*}
		\partial_{x_n}c(x^\prime,x_n) = \partial_{x_n}c(x^\prime,x_n) - \partial_{x_n}c(x^\prime,0) = \int_0^1\partial_{x_n}^2c(x^\prime,tx_n)x_n\:\text{d}t.
		\end{align*}
		This yields
		\begin{align*}
		&\;\;\;\;\:\Big|\int_{B_{\delta}(x^\prime)\times(0,\delta)}J_\varepsilon(|x-\hat{y}|)\partial_{x_n}c(x)(x_n+y_n)\text{d}y\Big| \\
		&= \Big|\int_{B_{\delta}(x^\prime)\times(0,\delta)}\int_0^1 J_\varepsilon(|x-\hat{y}|)\partial_{x_n}^2c(x^\prime,tx_n)x_n(x_n+y_n)\:\text{d}t\:\text{d}y\Big| \\
		&\leq \int_0^1\big|\partial_{x_n}^2c(x^\prime,tx_n)\big|\Big(\int_{B_{\delta}(x^\prime)\times(0,\delta)}J_\varepsilon(|x-\hat{y}|)(|x_n|+|y_n|)|x_n+y_n|\text{d}y\Big)\text{d}t \\ 
		&\leq 
		\int_0^1\big|\partial_{x_n}^2c(x^\prime,tx_n)\big|a_\varepsilon(x_n)\:\text{d}t, 
		\end{align*}
		where we defined 
		\begin{align*}
		a_\varepsilon(x_n) := \int_{B_{\delta}(x^\prime)\times(0,\delta)}\rho_\varepsilon(|x-\hat{y}|)\text{d}y.
		\end{align*}
		By construction and due to the properties of $\rho_\varepsilon$, the function $a_\varepsilon$ is in fact independent of $x^\prime$. Computing the $L^2$-norm of $a_\varepsilon$, we get
		\begin{align}\label{a_eps}
		\|a_\varepsilon\|_{L^2(\mathbb{R}_+)}^2 &= \int_{0}^\infty\left(\int_{B_{\delta}(x^\prime)\times(0,\delta)}\rho_\varepsilon(|x-\hat{y}|)\:\text{d}y\right)^2\text{d}x_n \nonumber\\
		&= \int_{0}^\infty\left(\int_{B_{\delta}(x^\prime)\times(0,\delta)}\varepsilon^{-n}\rho\Big(\Big|\frac{x-\hat{y}}{\varepsilon}\Big|\Big)\:\text{d}y\right)^2\text{d}x_n \nonumber\\ 
		&\leq \varepsilon\int_{0}^\infty\left(\int_{0}^\infty\int_{\mathbb{R}^{n-1}}\rho(|x-\hat{y}|)\:\text{d}y^\prime\text{d}y_n\right)^2\text{d}x_n \nonumber\\
		&\leq \varepsilon\int_{0}^R\left(\int_{0}^\infty\int_{\mathbb{R}^{n-1}}\rho(|x-\hat{y}|)\:\text{d}y^\prime\text{d}y_n\right)^2\text{d}x_n \leq R\varepsilon.
		\end{align}
		Here, we applied the transformations $y \mapsto \varepsilon y$ and $x_n \mapsto \varepsilon x_n$ and we used that $\text{supp}\rho\subset B_R(0)$. This then implies
		\begin{align*}
		&\int_{\mathbb{R}^n_+}\Big|\int_{B_{\delta}(x^\prime)\times(0,\delta)}J_\varepsilon(|x-\hat{y}|)\partial_{x_n}c(x)(x_n+y_n)\text{d}y\Big|^2\:\text{d}x \\
		&\leq 
		\int_{\mathbb{R}^n_+}|a_\varepsilon(x_n)|^2\left(\frac{1}{x_n}\int_0^{x_n}\|\partial_{x_n}^2c(x^\prime,.)\|_{L^\infty(\mathbb{R}_+)}\text{d}z_n\right)^2\text{d}x \\
		&\leq K\int_{\mathbb{R}^{n-1}}\int_{0}^\infty|a_\varepsilon(x_n)|^2\|\partial_{x_n}^2c(x^\prime,.)\|_{L^\infty(\mathbb{R}_+)}^2\text{d}x_n\text{d}x^\prime \\
		&\leq K\int_{\mathbb{R}^{n-1}}\|a_\varepsilon\|_{L^2(\mathbb{R}_+)}^2\|\partial_{x_n}^2c(x^\prime,.)\|_{H^1(\mathbb{R}_+)}^2\text{d}x^\prime 
		\leq K\varepsilon \|c\|_{H^3(\mathbb{R}^n_+)}^2,
		\end{align*}
		where we used the embedding $H^1(\mathbb{R}_+) \hookrightarrow L^\infty(\mathbb{R}_+)$.
		
		Ad $I_{\varepsilon,\delta}^3$: Here, it holds
		\begin{align}
		&\int_{\mathbb{R}^n_+}\Big|\int_{B_{\delta}(x^\prime)\times(0,\delta)}J_\varepsilon(|x-\hat{y}|)R_2(x,\hat{y})\text{d}y\Big|^2\text{d}x \nonumber\\
		&\leq K\int_{\mathbb{R}^n_+}\left(\int_{B_{\delta}(x^\prime)\times(0,\delta)}\int_{0}^1\rho_\varepsilon(|x-\hat{y}|)|D^2c(\hat{y}+t(x-\hat{y}))|\:\text{d}t\text{d}y\right)^2\text{d}x \nonumber \\
		&\leq K\int_{\mathbb{R}^n_+}\left(\int_{B_{\delta}(x^\prime)\times(0,\delta)}\rho_\varepsilon(|x-\hat{y}|)\:\text{d}y\right)\left(\int_{0}^1\int_{B_{\delta}(x^\prime)\times(0,\delta)}\rho_\varepsilon(|x-\hat{y}|)|D^2c(\hat{y}+t(x-\hat{y}))|^2\:\text{d}y\text{d}t\right)\:\text{d}x \nonumber\\
		&\leq K\int_{\mathbb{R}^n_+}a_\varepsilon(x_n)\left(\int_{0}^1\int_{\mathbb{R}^n_+}\rho_\varepsilon(|x-\hat{y}|)\|D^2c(\hat{y}^\prime+t(x^\prime-\hat{y}^\prime),.)\|_{L^\infty(\mathbb{R}_+)}^2\:\text{d}y\text{d}t\right)\:\text{d}x \nonumber\\
		&\leq K\int_{0}^R a_\varepsilon(x_n)\text{d}x_n\int_{0}^1\int_{\mathbb{R}^{n-1}}\int_{\mathbb{R}^{n-1}}\|\rho_\varepsilon(|x^\prime-\hat{y}^\prime,.|)\|_{L^1(\mathbb{R}_+)}\|D^2c(\hat{y}^\prime+t(x^\prime-\hat{y}^\prime),.)\|_{L^\infty(\mathbb{R}_+)}^2\:\text{d}y\text{d}x^\prime\text{d}t \nonumber\\
		&\leq K\left(\int_{0}^R a_\varepsilon(x_n)\:\text{d}x_n\right)\|c\|^2_{H^3(\mathbb{R}^n_+)} \leq K\varepsilon\|c\|^2_{H^3(\mathbb{R}^n_+)} \label{helpinequ1},
		\end{align}
		where we first used the inequality of Cauchy-Schwarz and then Lemma \ref{lemTraf} together with the embedding $H^1(\mathbb{R}_+) \hookrightarrow L^\infty(\mathbb{R}_+)$ in the fourth step. In the last step, we used $\text{supp}\rho\subset B_R(0)$ and applied the transformations $y \mapsto \varepsilon y$ and $x_n \mapsto \varepsilon x_n$.
		Altogether, this shows 
		\begin{align*}
		\|\mathcal{R}_\varepsilon c\|_{L^2(\mathbb{R}^n_+)} \leq K\sqrt{\varepsilon}\|c\|_{H^3(\mathbb{R}^n_+)}
		\end{align*}
		for all $c\in C^3_0(\overline{\mathbb{R}^n_+})$ with $\partial_\mathbf{n}c = 0$ on $\partial\mathbb{R}^n_+$. 
		
		In the next step, we want to use a denseness argument, in order conclude the proof. Let \linebreak $c\in H^3(\mathbb{R}^n_+)$ with $\partial_\mathbf{n}c = 0$ on $\partial\mathbb{R}^n_+$ be arbitrary. Since the space $\{c\in C^\infty_0(\overline{\mathbb{R}^n_+}): \partial_\mathbf{n}c = 0\text{ on }\partial\mathbb{R}^n_+\}$ is dense in $\{c\in H^3(\mathbb{R}^n_+): \partial_\mathbf{n}c = 0\text{ on }\partial\mathbb{R}^n_+\}$, cf. Remark \ref{remdense}, there exists a sequence $(c_k)_{k\in\mathbb{N}}\subset C^\infty_0(\overline{\mathbb{R}^n_+})$ with $\partial_\mathbf{n}c = 0$ on $\partial\mathbb{R}^n_+$ such that $c_k \rightarrow c$ in $H^3(\mathbb{R}^n_+)$. Thanks to our results so far, the sequence 
		\begin{align*}
		\Big(\mathcal{L}_\varepsilon^{\mathbb{R}^n_+} c_k +\Delta c_k\Big)_{k\in\mathbb{N}} \subset L^2(\mathbb{R}^n_+)
		\end{align*}
		is bounded. Thus, there exists a subsequence, which is again denoted by $\Big(\mathcal{L}_\varepsilon^{\mathbb{R}^n_+} c_k +\Delta c_k\Big)_{k\in\mathbb{N}}$, such that 
		\begin{align*}
		\mathcal{L}_\varepsilon^{\mathbb{R}^n_+} c_k +\Delta c_k \rightharpoonup w\;\;\;\text{in }L^2(\mathbb{R}^n_+)\;\;\;\text{as }k\rightarrow\infty
		\end{align*}
		for some $w\in L^2(\mathbb{R}^n_+)$. Since $\mathcal{L}_\varepsilon^{\mathbb{R}^n_+}$ is linear and continuous, and therefore weakly continuous, it follows that $w = \mathcal{L}_\varepsilon^{\mathbb{R}^n_+} c + \Delta c$. Then, the weak lower semi-continuity of norms implies
		\begin{align*}
		\Big\|\mathcal{L}_\varepsilon^{\mathbb{R}^n_+} c + \Delta c\Big\|_{L^2(\mathbb{R}^n_+)} \leq \liminf_{k\rightarrow\infty}\Big\|\mathcal{L}_\varepsilon^{\mathbb{R}^n_+} c_k + \Delta c_k\Big\|_{L^2(\mathbb{R}^n_+)} \leq \liminf_{k\rightarrow\infty}K\sqrt{\varepsilon}\|c_k\|_{H^2(\mathbb{R}^n_+)} \leq K\sqrt{\varepsilon}
		\end{align*}
		for all $c\in H^3(\mathbb{R}^n_+)$ with $\partial_\mathbf{n}c = 0$ on $\partial\mathbb{R}^n_+$ with $\|c\|_{H^3(\mathbb{R}^n_+)} \leq 1$. In particular, this concludes the proof.
	\end{proof}
%\newpage
\newtheorem{rem0}[lemma0]{Remark}
\begin{rem0}
	The rate of convergence obtained in Lemma \ref{lemma3} is optimal. Even in the simplest case, where $n=1$ and $c\in C^\infty_0(\overline{\mathbb{R}_+})$, we do not gain a better rate of convergence in $L^2(\mathbb{R}_+)$ unless $\partial_\mathbf{n}^2c=0$ on $\partial\mathbb{R}_+$. This shows the following calculation:
	Let $x>0$. Then, we obtain for the error term  
	\begin{align*}
	\mathcal{R}_\varepsilon c(x) &= \int_{0}^\infty J_\varepsilon(|x-\hat{y}|)(c(x)-c(\hat{y}))\:\textup{d}y \\
	&= \int_{B_{\delta}(x)\cap\mathbb{R}_+}J_\varepsilon(|x-\hat{y}|)(c(x)-c(\hat{y}))\:\textup{d}y + \int_{\mathbb{R}_+\setminus{B_{\delta}(x)}}J_\varepsilon(|x-\hat{y}|)(c(x)-c(\hat{y}))\:\textup{d}y \\
	&= \int_{B_{\delta}(x)\cap\mathbb{R}_+} J_\varepsilon(|x-\hat{y}|)\Big(\frac{1}{2}c^{\prime\prime}(0)(x-\hat{y})^2+R_3(x,\hat{y})\Big)\textup{d}y + \int_{\mathbb{R}_+\setminus{B_{\delta}(x)}}J_\varepsilon(|x-\hat{y}|)(c(x)-c(\hat{y}))\:\textup{d}y \\
	&= \frac{1}{2}c^{\prime\prime}(0)\int_{B_{\delta}(x)\cap\mathbb{R}_+}\rho_\varepsilon(|x-\hat{y}|)\textup{d}y + \mathcal{O}(\varepsilon),
	\end{align*}
	where $\int_{B_{\delta}(x)\cap\mathbb{R}_+}\rho_\varepsilon(|x-\hat{y}|)\textup{d}y =: a_\varepsilon(x)$ and $\|a_\varepsilon(x)\|_{L^2(\mathbb{R}_+)} \geq K\sqrt{\varepsilon}$ for $\varepsilon>0$ small enough similar as in \eqref{a_eps}.
	In the third step, we used a Taylor expansion. In the last step, we then applied Lemma \ref{lemma2}. Here, the term $\mathcal{O}(\varepsilon)$ is measured with respect to the $L^2(\mathbb{R}_+)$-norm.
\end{rem0}
	\newtheorem{cor3}[lemma0]{Corollary}
	\begin{cor3}\label{cor3}
		Let $c\in H^2(\mathbb{R}^n_+)$ with $\partial_\mathbf{n}c = 0$ on $\partial\mathbb{R}^n_+$. Then, it holds
		\begin{align*}
			\Big\|\mathcal{L}_\varepsilon^{\mathbb{R}^n_+} c + \Delta c\Big\|_{L^2(\mathbb{R}^n_+)} \rightarrow 0\;\;\;\text{as }\varepsilon\searrow 0.
		\end{align*}
	\end{cor3}
	\begin{proof}
		Thanks to Lemma \ref{lemma3}, it suffices to prove that $\big\|\mathcal{L}_\varepsilon^{\mathbb{R}^n_+}c + \Delta c\big\|_{L^2(\mathbb{R}^n_+)}$ is bounded uniformly in $\varepsilon>0$ for all $c\in H^2(\mathbb{R}^n_+)$ with $\partial_\mathbf{n}c = 0$ on $\partial\mathbb{R}^n_+$. Then, the Banach-Steinhaus Theorem concludes the proof.
		\newline
		In fact, we only need to estimate the error term \eqref{error} in a suitable way. First of all, we prove the assertion for functions $c\in C^2_0(\overline{\mathbb{R}^n_+})$ with $\partial_\mathbf{n}c = 0$ on $\partial\mathbb{R}^n_+$ and use a density argument afterwards to conclude the proof.
		In the term 
		$$\Big|\int_{\mathbb{R}^n_+\setminus \big(B_{\delta}(x^\prime)\times(0,\delta)\big)}J_\varepsilon(|x-\hat{y}|)(c(x)- c(\hat{y}))\text{d}y\Big|,$$ 
		we can apply Lemma \ref{lemma2}, again.  
		Thus, it suffices to consider the first part of \eqref{error}.
		As in the proof of Lemma \ref{lemma3}, it holds $I^1_{\varepsilon,\delta} = 0$. In $I^2_{\varepsilon,\delta}$, we use the properties of $c$ and the fundamental theorem of calculus to obtain
		\begin{align*}
		&\;\;\;\;\:\Big|\int_{B_{\delta}(x^\prime)\times(0,\delta)}J_\varepsilon(|x-\hat{y}|)\partial_{x_n}c(x)(x_n+y_n)\text{d}y\Big| \\
		&= \Big|\int_{B_{\delta}(x^\prime)\times(0,\delta)}\int_0^1 J_\varepsilon(|x-\hat{y}|)\partial_{x_n}^2c(x^\prime,tx_n)x_n(x_n+y_n)\:\text{d}t\text{d}y\Big| \\
		&\leq \int_{B_{\delta}(x^\prime)\times(0,\delta)}\int_0^1 J_\varepsilon(|x-\hat{y}|)|\partial_{x_n}^2c(x^\prime,tx_n)||x_n||x_n+y_n|\:\text{d}t\text{d}y \\
		&\leq \int_0^1|\partial_{x_n}^2c(x^\prime,tx_n)|\Big(\int_{B_{\delta}(x^\prime)\times(0,\delta)}J_\varepsilon(|x-\hat{y}|)(|x_n|+|y_n|)|x_n+y_n|\text{d}y\Big)\text{d}t. 
		\end{align*}
		In the inner integral, it holds
		\begin{align*}
		\int_{B_{\delta}(x^\prime)\times(0,\delta)}J_\varepsilon(|x-\hat{y}|)(|x_n|+|y_n|)|x_n+y_n|\text{d}y &\leq \int_{B_{\delta}(x^\prime)\times(0,\delta)}J_\varepsilon(|x-\hat{y}|)K(|x_n|^2+|y_n|^2)\text{d}y \\
		&\leq \int_{B_{\delta}(x^\prime)\times(0,\delta)}K\rho_\varepsilon(|x-\hat{y}|)\text{d}y \leq K.
		\end{align*}
		This implies
		\begin{align*}
		\Big|\int_{B_{\delta}(x^\prime)\times(0,\delta)}J_\varepsilon(|x-\hat{y}|)\partial_{x_n}c(x)(x_n+y_n)\text{d}y\Big| &\leq \int_0^1 K|\partial_{x_n}^2c(x^\prime,tx_n)|\text{d}t \\ 
		&= K\frac{1}{x_n}\int_0^{x_n}|\partial_{x_n}^2c(x^\prime,z_n)|\text{d}z_n,
		\end{align*}
		where we changed variables as $z_n = tx_n$ in the last step. This implies
		\begin{align}
			\int_{\mathbb{R}^n_+}\Big|\int_{B_{\delta}(x^\prime)\times(0,\delta)}J_\varepsilon(|x-\hat{y}|)\partial_{x_n}c(x)(x_n+y_n)\text{d}y\Big|^2\text{d}x &\leq K\int_{\mathbb{R}^n_+}\Big(\frac{1}{x_n}\int_0^{x_n}|\partial_{x_n}^2c(x^\prime,z_n)|\text{d}z_n\Big)^2\text{d}x \nonumber\\
			%&= \int_{\mathbb{R}^{n-1}}\int_{0}^\infty\Big(\frac{1}{x_n}\int_0^{x_n}|\partial_{x_n}^2c(x^\prime,z_n)|\text{d}z_n\Big)^2\text{d}x_n\text{d}x^\prime \\
			&\leq K\int_{\mathbb{R}^{n-1}}\int_{0}^\infty|\partial_{x_n}^2c(x^\prime,x_n)|^2\text{d}x_n\text{d}x^\prime \nonumber\\
			&\leq K\|c\|_{H^2(\mathbb{R}^n_+)}^2, \label{auxest2}
		\end{align}
		where we applied Hardy`s inequality.
 	
 	Ad $I_{\varepsilon,\delta}^3$: Here, we use Fubini`s Theorem and Lemma \ref{lemTraf} to get
 	\begin{align}
 	&\int_{\mathbb{R}^n_+}\Big|\int_{B_{\delta}(x^\prime)\times(0,\delta)}J_\varepsilon(|x-\hat{y}|)R_2(x,\hat{y})\text{d}y\Big|^2\text{d}x \nonumber\\
 	&\leq K\int_{\mathbb{R}^n_+}\int_{B_{\delta}(x^\prime)\times(0,\delta)}\int_{0}^1\rho_\varepsilon(|x-\hat{y}|)|D^2c(\hat{y}+t(x-\hat{y}))|^2\text{d}t\text{d}y\text{d}x \nonumber\\
 	&\leq K\int_{0}^1\int_{\mathbb{R}^{n-1}}\int_{\mathbb{R}^{n-1}}\|\rho_\varepsilon(|(x^\prime-y^\prime,.)|)\|_{L^1(\mathbb{R})}\|D^2c((1-t)y^\prime+tx^\prime,.)\|_{L^2(\mathbb{R}_+)}^2\text{d}y^\prime\text{d}x^\prime\text{d}t \nonumber\\
 	&\leq K\|c\|_{H^2(\mathbb{R}^n_+)}^2. \label{auxest1}
 	\end{align}
 	Altogether, we get the following estimate
 	\begin{align*}
 	\|\mathcal{R}_\varepsilon\tilde{c}\|_{L^2(\mathbb{R}^n_+)} \leq K_\delta\|c\|_{H^2(\mathbb{R}^n_+)} + K\|c\|_{H^2(\mathbb{R}^n_+)}
 	\end{align*}
 	for all $c\in C^2_0(\overline{\mathbb{R}^n_+})$ with $\partial_{x_n}c_{\big|\{x_n=0\}} = 0$. %Here, $\tilde{c}\in C^2_0(\mathbb{R}^n)$ denotes an extension of $c$ to the whole of $\mathbb{R}^n$, cf. the proof of Lemma \ref{lemma3}.
 	Finally, a density argument as in the proof of Lemma \ref{lemma3} finishes the proof.
	\end{proof}
	In the next step, we prove convergence on the bent half-space.
	\newtheorem{lemma4}[lemma0]{Lemma}
	\begin{lemma4}\label{lemma4}
		Let $\gamma\in C^3_b(\mathbb{R}^{n-1})$ with $\|\gamma\|_{C^3_b(\mathbb{R}^{n-1})}$ sufficiently small and $c\in H^3(\mathbb{R}^n_\gamma)$ with $\partial_\mathbf{n}c = 0$ on $\partial\mathbb{R}^n_\gamma$. Then, it holds
		\begin{align*}
		\Big\|\mathcal{L}_\varepsilon^{\mathbb{R}^n_\gamma} c + \Delta c\Big\|_{L^2(\mathbb{R}^n_\gamma)} \leq K\sqrt{\varepsilon}\|c\|_{H^3(\mathbb{R}^n_\gamma)}.
		\end{align*}
	\end{lemma4}
	\begin{proof}
		Let $\tilde{c}\in H^3(\mathbb{R}^n)$ denote an extension of $c$ to the whole of $\mathbb{R}^n$ such that $\|\tilde{c}\|_{H^3(\mathbb{R}^n)} \leq K\|c\|_{H^3(\mathbb{R}^n_\gamma)}$. Then, we have
		\begin{align*}
			\Big\|\mathcal{L}_\varepsilon^{\mathbb{R}^n_\gamma}c+\Delta c\Big\|_{L^2(\mathbb{R}^n_\gamma)} &\leq \Big\|\mathcal{L}_\varepsilon^{\mathbb{R}^n}\tilde{c}+\Delta \tilde{c}\Big\|_{L^2(\mathbb{R}^n)} + \|\mathcal{R}_\varepsilon\tilde{c}\|_{L^2(\mathbb{R}^n_\gamma)} \\
			&\leq K\varepsilon\|c\|_{H^3(\mathbb{R}^n_\gamma)} + \|\mathcal{R}_\varepsilon\tilde{c}\|_{L^2(\mathbb{R}^n_\gamma)},
		\end{align*}
		where we used Lemma \ref{lemma3} and defined the error term
		\begin{align}\label{error1}
			\mathcal{R}_\varepsilon\tilde{c}(x) := \int_{(\mathbb{R}^n_\gamma)^c}J_\varepsilon(|x-y|)\big(c(x) - \tilde{c}(y)\big)\;\text{d}y
		\end{align}
		for a.e. $x\in\mathbb{R}^n_\gamma$. Thus, it suffices to consider the error term. 
		\newline
		Since $\gamma\in C^3_b(\mathbb{R}^{n-1})$, there exists a $C^{2,1}-$diffeomorphism $F_\gamma : \mathbb{R}^n \rightarrow \mathbb{R}^n$ satisfying $F_\gamma(\mathbb{R}^n_+) = \mathbb{R}^n_\gamma$, $F_\gamma(x^\prime,0) = (x^\prime,\gamma(x^\prime))$ and $-\partial_{x_n}F_\gamma(x)|_{x_n=0} = \mathbf{n}(x^\prime,\gamma(x^\prime))$, where $\mathbf{n}$ denotes the exterior unit normal on $\partial\mathbb{R}^n_\gamma$, cf. \cite[Lemma 2.1]{Schumacher}. 
		
		In the following, we assume
		\begin{align}\label{ass1}
		\sup_{\hat{x}\in\mathbb{R}^n}|\text{D}F_\gamma(\hat{x}) - \text{Id}| \leq \alpha
		\end{align}
		for some $\alpha\in(0,\frac{1}{3})$. Using the diffeomorphism $F_\gamma$ and reflection afterwards, we compute  
		\begin{align*}
			\mathcal{R}_\varepsilon \tilde{c}(x)	
			&= \int_{\mathbb{R}^n_-}J_\varepsilon(|F_\gamma(\hat{x})-F_\gamma(\hat{y})|)\big(c(F_\gamma(\hat{x})) - \tilde{c}(F_\gamma(\hat{y}))\big)\big|\det(\text{D}F_\gamma(\hat{y}))\big|\:\text{d}y \\ 
			&= \int_{\mathbb{R}^n_+}J_\varepsilon(|A_\gamma(\hat{x},\bar{y})(\hat{x}-\bar{y})|)\big(u(\hat{x})-u(\bar{y})\big)\big|\det(\text{D}F_\gamma(\bar{y}))\big|\:\text{d}y
		\end{align*}
		for almost all $x\in\mathbb{R}^n_\gamma$. Here, we defined $u:= c\circ F_\gamma:\mathbb{R}^n_+ \rightarrow \mathbb{R}$, $F_\gamma(\hat{x}) = x$, $\bar{y} = (y_1,\ldots,y_{n-1},-y_n)$ and we used 
		\begin{align*}
		F_\gamma(\hat{x}) - F_\gamma(\bar{y}) = \int_{0}^1\text{D}F_\gamma(\bar{y} + t(\hat{x}-\bar{y}))(\hat{x}-\bar{y})\:\text{d}t =: A_\gamma(\hat{x},\bar{y})(\hat{x}-\bar{y}).
		\end{align*}
		Next, we rewrite the error term $\mathcal{R}_\varepsilon$ as 
		\begin{align}
		\mathcal{R}_\varepsilon c(x) &= \int_{B_{\delta}(\hat{x}^\prime)\times(0,\delta)}J_\varepsilon(|A_\gamma(\hat{x},\hat{x})(\hat{x}-\bar{y})|)\big(u(\hat{x})-u(\bar{y})\big)\big|\det\text{D}F_\gamma(\hat{x})\big|\:\text{d}y \nonumber\\
		&+ \int_{B_{\delta}(\hat{x}^\prime)\times(0,\delta)}K_\varepsilon(|A_\gamma(\hat{x},\hat{x})(\hat{x}-\bar{y})|)\big(u(\hat{x})-u(\bar{y})\big)\:\text{d}y\nonumber\\
		&+\int_{\mathbb{R}^n_+\setminus\big(B_{\delta}(\hat{x}^\prime)\times(0,\delta)\big)}J_\varepsilon(|A_\gamma(\hat{x},\bar{y})(\hat{x}-\bar{y})|)\big(u(\hat{x})-u(\bar{y})\big)\big|\det(\text{D}F_\gamma(\bar{y}))\big|\:\text{d}y \nonumber\\
		&=: I_\varepsilon^1(x) + I_\varepsilon^2(x) + I_\varepsilon^3(x)\label{error2},
		\end{align}
		where $K_\varepsilon(|A_\gamma(\hat{x},\hat{x})(\hat{x}-\bar{y})|)$ is defined by
		\begin{align*}
		K_\varepsilon(|A_\gamma(\hat{x},\hat{x})(\hat{x}-\bar{y})|) := J_\varepsilon(|A_\gamma(\hat{x},\hat{y})(\hat{x}-\bar{y})|)\big|\det\text{D}F_\gamma(\bar{y})\big| - J_\varepsilon(|A_\gamma(\hat{x},\hat{x})(\hat{x}-\bar{y})|)\big|\det\text{D}F_\gamma(\hat{x})\big|
		\end{align*}
		for almost all $\hat{x},\bar{y}\in\mathbb{R}^n_+$.
		In $I_\varepsilon^3(x)$, we can apply Lemma \ref{lemma2}, since it holds 
		\begin{align*}
		|A_\gamma(\hat{x},\bar{y})(\hat{x}-\bar{y})| \geq K|\hat{x}-\bar{y}| \geq K\delta.	
		\end{align*} 
		In the next step, we estimate the terms $I_\varepsilon^1(x)$ and $ I_\varepsilon^2(x)$ separately. In the following, let $c\in C^{2,1}_0(\overline{\mathbb{R}^n_\gamma})$ with $\partial_\mathbf{n}c = 0$ on $\partial\mathbb{R}^n_\gamma$.
		
		Ad $I_\varepsilon^1(x)$: First, we use a Taylor expansion for $u$ to get
		\begin{align*}
		I_\varepsilon^1(x) &= \int_{B_{\delta}(\hat{x}^\prime)\times(0,\delta)}J_\varepsilon(|A_\gamma(\hat{x},\hat{x})(\hat{x}-\bar{y})|)\nabla_{x^\prime}u(\hat{x})\cdot(\hat{x}^\prime-\bar{y}^\prime)\big|\det \text{D}F_\gamma(\hat{x})\big|\;\text{d}y \\
		&+ \int_{B_{\delta}(\hat{x}^\prime)\times(0,\delta)}J_\varepsilon(|A_\gamma(\hat{x},\hat{x})(\hat{x}-\bar{y})|)\partial_{x_n}u(\hat{x})(\hat{x}_n+\bar{y}_n)\big|\det \text{D}F_\gamma(\hat{x})\big|\;\text{d}y \\ 
		&+ \int_{B_{\delta}(\hat{x}^\prime)\times(0,\delta)}J_\varepsilon(|A_\gamma(\hat{x},\hat{x})(\hat{x}-\bar{y})|)R_2(\hat{x},\bar{y})\big|\det \text{D}F_\gamma(\hat{x})\big|\;\text{d}y,
		\end{align*}
		where the error term $R_2(\hat{x},\bar{y})$ is defined by
		\begin{align*}
			R_2(\hat{x},\bar{y}) := \sum_{|\beta|=2}\frac{2}{\beta!}\left(\int_{0}^1(1-t)\text{D}^\beta c(\bar{y}+t(\hat{x}-\bar{y}))\text{d}t\right)(\hat{x}-\bar{y})^\beta.
		\end{align*}
		Observe that by our choice of $F_\gamma$, the term $A_\gamma(\hat{x},\hat{x})$ is given by
		\begin{align*}
		A_\gamma(\hat{x},\hat{x}) = \text{D}F_\gamma(\hat{x}) = U(\hat{x})\left(\begin{array}{ccc|c}
		& & & 0 \\
		& A^\prime(\hat{x}) & & \vdots \\
		& & & 0 \\ \hline
		0 & \ldots & 0 & 1
		\end{array}
		\right),
		\end{align*}
		where $U(\hat{x})\in \text{SO}(n)$ and $A^\prime(\hat{x})\in\mathbb{R}^{(n-1)\times(n-1)}$, cf. \cite{AbelsTerasawa} for details. This implies
		\begin{align*}
		|A_\gamma(\hat{x},\hat{x})(\hat{x}-\bar{y})|^2 = |A^\prime(\hat{x})(\hat{x}^\prime-y^\prime)|^2 + |\hat{x}_n+y_n|^2,
		\end{align*}
		and therefore it follows that the integrand in the first term of $I_\varepsilon^1(x)$ is odd with respect to $\hat{x}^\prime - \bar{y}^\prime$. Consequently, 
		\begin{align*}
		\int_{B_{\delta}(\hat{x}^\prime)\times(0,\delta)}J_\varepsilon(|A_\gamma(\hat{x},\hat{x})(\hat{x}-\bar{y})|)\nabla_{x^\prime}u(\hat{x})\cdot(\hat{x}^\prime-\bar{y}^\prime)\big|\det \text{D}F_\gamma(\hat{x})\big|\;\text{d}y = 0.
		\end{align*}
		Furthermore, $F_\gamma$ satisfies
		\begin{align*}
		\partial_{x_n}u(\hat{x})_{\big|\{x_n=0\}} = \partial_{x_n}\big(c \circ F_\gamma\big)(\hat{x})_{\big|\{x_n=0\}} = \nabla c(x^\prime,\gamma(x^\prime))\cdot(-\mathbf{n}(x^\prime,\gamma(x^\prime))) = 0,
		\end{align*}
		since $\partial_\mathbf{n}c = 0$ on $\partial\mathbb{R}^n_\gamma$. Thus, it holds
		\begin{align*}
		\partial_{x_n}u(\hat{x}) = \partial_{x_n}u(\hat{x}) - \partial_{x_n}u(\hat{x}^\prime,0) = \Big(\int_{0}^1\partial_{x_n}^2u(\hat{x}^\prime,tx_n)\:\text{d}t\Big)x_n.
		\end{align*}
		Hence, we obtain in the second term of $I_\varepsilon^1(x)$
		\begin{align*}
		&\Big|\int_{B_{\delta}(\hat{x}^\prime)\times(0,\delta)}J_\varepsilon(|A_\gamma(\hat{x},\hat{x})(\hat{x}-\bar{y})|)\partial_{x_n}u(\hat{x})(\hat{x}_n+\bar{y}_n)\big|\det \text{D}F_\gamma(\hat{x})\big|\;\text{d}y\Big| \\
		&\leq \int_{0}^1a_\varepsilon(x_n)|\partial_{x_n}^2u(\hat{x}^\prime,t\hat{x}_n)|\:\text{d}t,
		\end{align*}
		where we used the inequalities $|A_\gamma(\hat{x},\bar{y})(\hat{x}-\bar{y})| \geq K|\hat{x}-\bar{y}|$ for all $\hat{x},\bar{y}\in\mathbb{R}^n$ and $\big|\det \text{D}F_\gamma(\hat{x})\big|\leq K$ for all $\hat{x}\in\mathbb{R}^n$.
		Here, we defined
		\begin{align*}
			a_\varepsilon(x_n) := \sup_{x^\prime\in\mathbb{R}^{n-1}}\int_{B_{\delta}(\hat{x}^\prime)\times(0,\delta)}\rho_\varepsilon(|A_\gamma(\hat{x},\hat{x})(\hat{x}-\bar{y})|)\;\text{d}y.
		\end{align*}
		Note that it holds $a_\varepsilon\in L^2(\mathbb{R}_+)$ and $\|a_\varepsilon\|_{L^2(\mathbb{R}_+)}\leq K\sqrt{\varepsilon}$. This follows from the same calculation as in the proof of Lemma \ref{lemma3}.
		Computing the $L^2$-norm, we then obtain
		\begin{align*}
			&\int_{\mathbb{R}^n_+}\Big|\int_{B_{\delta}(\hat{x}^\prime)\times(0,\delta)}J_\varepsilon(|A_\gamma(\hat{x},\hat{x})(\hat{x}-\bar{y})|)\partial_{x_n}u(\hat{x})(\hat{x}_n+y_n)\big|\det \text{D}F_\gamma(\hat{x})\big|\;\text{d}y\Big|^2\text{d}\hat{x} \\
			&\leq \int_{\mathbb{R}^n_+}|a_\varepsilon(\hat{x}_n)|^2\Big(\frac{1}{\hat{x}_n}\int_{0}^{\hat{x}_n}|\partial_{\hat{x}_n}u(\hat{x}^\prime,z_n)|\text{d}z_n\Big)^2\text{d}\hat{x} \\
			&\leq  \int_{\mathbb{R}^{n-1}}\int_{0}^\infty|a_\varepsilon(\hat{x}_n)|^2\|\partial_{\hat{x}_n}u(\hat{x}^\prime,.)\|_{L^\infty(\mathbb{R})}^2\text{d}\hat{x}_n\text{d}\hat{x}^\prime \leq K\varepsilon\|u\|_{H^3(\mathbb{R}^n_+)}^2.
		\end{align*}
		 
		In the third term of $I_\varepsilon^1(x)$, we have
		\begin{align}
		&\int_{\mathbb{R}^n_+}\Big|\int_{B_{\delta}(\hat{x}^\prime)\times(0,\delta)}J_\varepsilon(|A_\gamma(\hat{x},\hat{x})(\hat{x}-\bar{y})|)R_2(\hat{x},\bar{y})\big|\det \text{D}F_\gamma(\hat{x})\big|\;\text{d}y\Big|^2\:\text{d}\hat{x} \nonumber\\
		&\leq K\int_{\mathbb{R}^n_+} a_\varepsilon(\hat{x}_n)\left(\int_{0}^1\int_{B_{\delta}(\hat{x}^\prime)\times(0,\delta)}\rho_\varepsilon(|A_\gamma(\hat{x},\hat{x})(\hat{x}-\bar{y})|)|D^2u(\bar{y}+t(\hat{x}-\bar{y}))|^2\:\text{d}y\text{d}t\right)\text{d}\hat{x} \nonumber \\
		&\leq K\left(\int_{0}^Ra_\varepsilon(\hat{x}_n)\:\text{d}\hat{x}_n\right)\|u\|_{H^3(\mathbb{R}^n_+)}^2 
		\leq K\varepsilon\|u\|_{H^3(\mathbb{R}^n_+)}^2 \label{auxest},
		\end{align}
		where we used the calculations as in \eqref{helpinequ1} and Lemma \ref{lemTraf}.
		Altogether, we end up with 
		\begin{align*}
		\big\|I^1_\varepsilon\big\|_{L^2(\mathbb{R}^n_\gamma)} \leq K\sqrt{\varepsilon}\|c\|_{H^3(\mathbb{R}^n_\gamma)},
		\end{align*}
		where we used the same estimates as in the proof of Lemma \ref{lemma3}. 
		
		Ad $I_\varepsilon^2(x)$: Here, we first rewrite the term $K_\varepsilon(|A_\gamma(\hat{x},\hat{x})(\hat{x}-\hat{y})|)$ in the following way:
		\begin{align*}
		K_\varepsilon(|A_\gamma(\hat{x},\hat{x})(\hat{x}-\bar{y})|) &= J_\varepsilon(|A_\gamma(\hat{x},\bar{y})(\hat{x}-\bar{y})|)\big(\big|\det\text{D}F_\gamma(\bar{y})\big|-\big|\det\text{D}F_\gamma(\hat{x})\big|\big) \\
		&+ \Big(J_\varepsilon(|A_\gamma(\hat{x},\bar{y})(\hat{x}-\bar{y})|) - J_\varepsilon(|A_\gamma(\hat{x},\hat{x})(\hat{x}-\bar{y})|)\Big)\big|\det\text{D}F_\gamma(\hat{x})\big|.
		\end{align*}
		Using this identity, we then get
		\begin{align*}
		|I_\varepsilon^2(x)| &\leq \Big|\int_{B_{\delta}(\hat{x}^\prime)\times(0,\delta)}J_\varepsilon(|A_\gamma(\hat{x},\bar{y})(\hat{x}-\bar{y})|)\big(\big|\det\text{D}F_\gamma(\bar{y})\big|-\big|\det\text{D}F_\gamma(\hat{x})\big|\big)\big(u(\hat{x})-u(\bar{y})\big)\;\text{d}y\Big| \\
		&+ \Big|\int_{B_{\delta}(\hat{x}^\prime)\times(0,\delta)}\Big(J_\varepsilon(|A_\gamma(\hat{x},\bar{y})(\hat{x}-\bar{y})|) - J_\varepsilon(|A_\gamma(\hat{x},\hat{x})(\hat{x}-\bar{y})|)\Big)\big|\det\text{D}F_\gamma(\hat{x})\big|\big(u(\hat{x})-u(\bar{y})\big)\;\text{d}y\Big|.
		\end{align*}
		In the first term on the right-hand side, we use the continuity of $|\det\text{D}F_\gamma|$, the fundamental theorem of calculus as well as the Cauchy-Schwarz inequality and the calculations in $I^1_\varepsilon(x)$ to estimate the $L^2(\mathbb{R}^n_+)$-norm
		\begin{align*}
		&\int_{\mathbb{R}^n_+}\Big|\int_{B_{\delta}(\hat{x}^\prime)\times(0,\delta)}J_\varepsilon(|A_\gamma(\hat{x},\bar{y})(\hat{x}-\bar{y})|)\big(\big|\det\text{D}F_\gamma(\bar{y})\big|-\big|\det\text{D}F_\gamma(\hat{x})\big|\big)\big(u(\hat{x})-u(\bar{y})\big)\;\text{d}y\Big|^2\text{d}\hat{x} \\
		&\leq K\varepsilon\int_{\mathbb{R}^n_+}\int_{B_{\delta}(\hat{x}^\prime)\times(0,\delta)}\int_{0}^1\rho_\varepsilon(|A_\gamma(\hat{x},\bar{y})(\hat{x}-\bar{y})|)|Du(\bar{y}+t(\hat{x}-\bar{y}))|^2\:\text{d}t\text{d}y\text{d}\hat{x} 
		\leq K\varepsilon\|u\|_{H^3(\mathbb{R}^n_+)}^2,
		\end{align*}
		where we used the same arguments as in \eqref{auxest} in the last step.
		In the second term on the right-hand side, we use the mean value theorem to conclude
		\begin{align}
		\big|J_\varepsilon(|A_\gamma(\hat{x},\bar{y})(\hat{x}-\bar{y})|) - J_\varepsilon(|A_\gamma(\hat{x},\hat{x})(\hat{x}-\bar{y})|)\big| &= \left|\int_{0}^1\nabla J_\varepsilon(|z|)_{\big|z=z_t}\cdot\big(A_\gamma(\hat{x},\bar{y})-A_\gamma(\hat{x},\hat{x})\big)(\hat{x}-\bar{y})\:\text{d}t\right| \nonumber\\
		&\leq  \int_{0}^1\frac{1}{\varepsilon}\frac{\rho_\varepsilon^\prime(|z|)}{|z|^2}_{\big|z=z_t}\:K|\hat{x}-\bar{y}|^2\:\text{d}t \nonumber\\
		&+ \int_{0}^1\frac{\rho_\varepsilon(|z|)}{|z|^3}_{\big|z=z_t}\:K|\hat{x}-\bar{y}|^2\:\text{d}t,  \label{calc1}
		\end{align}
		where we defined 
		\begin{align*}
		z_t &:= M_t(\hat{x},\bar{y})(\hat{x}-\bar{y}), \\
		M_t(\hat{x},\bar{y}) &:= A_\gamma(\hat{x},\hat{x}) + t\big(A_\gamma(\hat{x},\bar{y}) - A_\gamma(\hat{x},\hat{x})\big)
		\end{align*}
		for all $\hat{x},\bar{y}\in\mathbb{R}^n$ and all $t\in[0,1]$. Thanks to assumption \eqref{ass1}, it follows that $|M_t(\hat{x},\bar{y})-\text{Id}|\leq 3\alpha$ for all $\hat{x},\bar{y}\in\mathbb{R}^n$ and all $t\in[0,1]$. Therefore, $M_t(\hat{x},\bar{y})^{-1}$ exists for all $\hat{x},\bar{y}\in\mathbb{R}^n$ and all $t\in[0,1]$ and satisfies the inequality
		\begin{align*}
		|M_t(\hat{x},\bar{y})^{-1}| \leq \frac{1}{1-|M_t(\hat{x},\bar{y})-\text{Id}|} < \frac{1}{1-3\alpha}
		\end{align*}
		for all $\hat{x},\bar{y}\in\mathbb{R}^n$ and all $t\in[0,1]$. Using \eqref{calc1}, we then obtain for the second term of $I_\varepsilon^2(x)$
		\begin{align*}
		&\int_{\mathbb{R}^n_+}\Big|\int_{B_{\delta}(\hat{x}^\prime)\times(0,\delta)}\Big(J_\varepsilon(|A_\gamma(\hat{x},\bar{y})(\hat{x}-\bar{y})|) - J_\varepsilon(|A_\gamma(\hat{x},\hat{x})(\hat{x}-\bar{y})|)\Big)\big|\det\text{D}F_\gamma(\hat{x})\big|\big(u(\hat{x})-u(\bar{y})\big)\;\text{d}y\Big|^2\:\text{d}\hat{x} \\
		&\leq K\int_{\mathbb{R}^n_+}\left(\int_{0}^1\int_{B_{\delta}(\hat{x}^\prime)\times(0,\delta)}\int_{0}^1\left(\frac{1}{\varepsilon}\frac{\rho_\varepsilon^\prime(|z_t|)}{|z_t|^2}+\frac{\rho_\varepsilon(|z_t|)}{|z_t|^3}\right)\big|Du(\bar{y}+s(\hat{x}-\bar{y}))\big||\hat{x}-\bar{y}|^3\:\text{d}s\text{d}y\text{d}t\right)^2\text{d}\hat{x} \\
		&\leq K\int_{\mathbb{R}^n_+}\left(\int_{0}^1\int_{B_{\delta}(\hat{x}^\prime)\times(0,\delta)}\int_{0}^1\Big(\frac{1}{\varepsilon}\rho_\varepsilon^\prime(|z_t|)|\hat{x}-\bar{y}|+\rho_\varepsilon(|z_t|)\Big)\big|Du(\bar{y}+s(\hat{x}-\bar{y}))\big|\:\text{d}s\text{d}y\text{d}t\right)^2\text{d}\hat{x} \\		  
		&\leq K\varepsilon\int_{0}^1\int_{0}^1\int_{\mathbb{R}^n_+}\int_{\mathbb{R}^n_+}\Big(\frac{1}{\varepsilon}\rho_\varepsilon^\prime(|z_t|)|\hat{x}-\bar{y}|+\rho_\varepsilon(|z_t|)\Big)\big|Du(\bar{y}+s(\hat{x}-\bar{y}))\big|^2\:\text{d}y\text{d}\hat{x}\text{d}s\text{d}t \leq K\varepsilon\|u\|_{H^3(\mathbb{R}^n_+)}^2,
		\end{align*}
 		where we used the same arguments as above.
		Then, we conclude 
		\begin{align*}
			\big\|I^2_\varepsilon\big\|_{L^2(\mathbb{R}^n_\gamma)} \leq K\sqrt{\varepsilon}\|c\|_{H^3(\mathbb{R}^n_\gamma)}.
		\end{align*}
		In particular, this implies
		\begin{align*}
			\Big\|\mathcal{L}_\varepsilon^{\mathbb{R}^n_\gamma} c + \Delta c\Big\|_{L^2(\mathbb{R}^n_\gamma)} \leq K\sqrt{\varepsilon}\|c\|_{H^3(\mathbb{R}^n_\gamma)}
		\end{align*}
		for all $c\in C^{2,1}_0(\overline{\mathbb{R}^n_\gamma})$ with $\partial_\mathbf{n}c = 0$ on $\partial\mathbb{R}^n_\gamma$. Finally, a density argument, cf. Lemma \ref{lemma3}, yields 
		\begin{align*}
		\Big\|\mathcal{L}_\varepsilon^{\mathbb{R}^n_\gamma} c + \Delta c\Big\|_{L^2(\mathbb{R}^n_\gamma)} \leq K\sqrt{\varepsilon}\|c\|_{H^3(\mathbb{R}^n_\gamma)}
		\end{align*}
		for all $c\in H^3(\mathbb{R}^n_\gamma)$ with $\partial_\mathbf{n}c = 0$ on $\partial\mathbb{R}^n_\gamma$. This concludes the proof.
	\end{proof}
	\newtheorem{cor4}[lemma0]{Corollary}
	\begin{cor4}\label{cor4}
		Let $c\in H^2(\mathbb{R}^n_\gamma)$ with $\partial_{\mathbf{n}}c = 0$ on $\partial\mathbb{R}^n_\gamma$. Then, it holds
		\begin{align*}
			\Big\|\mathcal{L}_\varepsilon^{\mathbb{R}^n_\gamma} c + \Delta c\Big\|_{L^2(\mathbb{R}^n_\gamma)} \rightarrow 0 \;\;\;\text{as }\varepsilon\searrow 0.
		\end{align*}
	\end{cor4}
	\begin{proof}
		Due to Lemma \ref{lemma4}, it suffices to show that $\big\|\mathcal{L}_\varepsilon^{\mathbb{R}^n_\gamma} c + \Delta c\big\|_{L^2(\mathbb{R}^n_\gamma)}$ is bounded uniformly in $\varepsilon>0$ for all $c\in H^2(\mathbb{R}^n_\gamma)$ with $\partial_{\mathbf{n}}c = 0$ on $\partial\mathbb{R}^n_\gamma$. Then, we can apply the Banach Steinhaus Theorem, which implies the assertion.
		\newline
		In fact, it suffices to bound the error term \eqref{error1} in a suitable way. Using the same methods as in the proof of Lemma \ref{lemma4}, we can rewrite $\mathcal{R}_\varepsilon$ as in \eqref{error2}. First of all, we prove the assertion for functions $c\in C^{1,1}_0(\overline{\mathbb{R}^n_\gamma})$ with $\partial_{\mathbf{n}}c = 0$ on $\partial\mathbb{R}^n_\gamma$ and use a density argument afterwards to conclude the proof. Observe that it suffices to integrate in $\mathcal{R}_\varepsilon$ only over $B_{\delta}(\hat{x}^\prime)\times(0,\delta)$, since otherwise we can employ Lemma \ref{lemma2}, cf. proof of Lemma \ref{lemma4}.
		
		Ad $I^1_\varepsilon(x)$: Using a Taylor expansion, we obtain 
		\begin{align*}
			I_\varepsilon^1(x) &= \int_{B_{\delta}(\hat{x}^\prime)\times(0,\delta)}J_\varepsilon(|A_\gamma(\hat{x},\hat{x})(\hat{x}-\bar{y})|)\nabla_{x^\prime}u(\hat{x})\cdot(\hat{x}^\prime-\bar{y}^\prime)\big|\det \text{D}F_\gamma(\hat{x})\big|\;\text{d}y \\
			&+ \int_{B_{\delta}(\hat{x}^\prime)\times(0,\delta)}J_\varepsilon(|A_\gamma(\hat{x},\hat{x})(\hat{x}-\bar{y})|)\partial_{x_n}u(\hat{x})(\hat{x}_n+\bar{y}_n)\big|\det \text{D}F_\gamma(\hat{x})\big|\;\text{d}y \\ 
			&+ \int_{B_{\delta}(\hat{x}^\prime)\times(0,\delta)}J_\varepsilon(|A_\gamma(\hat{x},\hat{x})(\hat{x}-\bar{y})|)R_2(\hat{x},\bar{y})\big|\det \text{D}F_\gamma(\hat{x})\big|\;\text{d}y.
		\end{align*}
	As in the proof before, the first term on the right-hand side vanishes, since the integrand is odd with respect to $\hat{x}^\prime-\bar{y}^\prime$. In the second term, the properties of $F_\gamma$ imply 
	\begin{align*}
	\partial_{x_n}u(\hat{x})_{\big|\{x_n=0\}} = \partial_{x_n}\big(c \circ F_\gamma\big)(\hat{x})_{\big|\{x_n=0\}} = \nabla c(x^\prime,\gamma(x^\prime))\cdot(-\mathbf{n}(x^\prime,\gamma(x^\prime))) = 0,
	\end{align*}
	since $\partial_\mathbf{n}c = 0$ on $\partial\mathbb{R}^n_\gamma$, and therefore
	\begin{align*}
	\partial_{x_n}u(\hat{x}) = \partial_{x_n}u(\hat{x}) - \partial_{x_n}u(\hat{x}^\prime,0) = \Big(\int_{0}^1\partial_{x_n}^2u(\hat{x}^\prime,tx_n)\:\text{d}t\Big)x_n.
	\end{align*}
	This yields
	\begin{align*}
	 &\Big|\int_{B_{\delta}(\hat{x}^\prime)\times(0,\delta)}J_\varepsilon(|A_\gamma(\hat{x},\hat{x})(\hat{x}-\bar{y})|)\partial_{x_n}u(\hat{x})(\hat{x}_n+\bar{y}_n)\big|\det \text{D}F_\gamma(\hat{x})\big|\;\text{d}y\Big| \\
	 &\leq K\int_{0}^1|\partial_{x_n}^2u(\hat{x}^\prime,t\hat{x}_n)|\Big(\int_{B_{\delta}(\hat{x}^\prime)\times(0,\delta)}\rho_\varepsilon(|A_\gamma(\hat{x},\hat{x})(\hat{x}-\bar{y})|)\;\text{d}y\Big)\:\text{d}t, \\
	 &\leq K\Big(\frac{1}{\hat{x}_n}\int_{0}^{\hat{x}_n}|\partial_{x_n}^2u(\hat{x}^\prime,z_n)|\;\text{d}z_n\Big)\Big(\int_{B_{\delta}(\hat{x}^\prime)\times(0,\delta)}\rho_\varepsilon(|A_\gamma(\hat{x},\hat{x})(\hat{x}-\bar{y})|)\;\text{d}y\Big)
 	\end{align*}
	where we used the inequalities $|A_\gamma(\hat{x},\bar{y})(\hat{x}-\bar{y})| \geq C|\hat{x}-\bar{y}|$ for all $\hat{x},\bar{y}\in\mathbb{R}^n$ and $\big|\det \text{D}F_\gamma(\hat{x})\big|\leq C$ for all $\hat{x}\in\mathbb{R}^n$.
	Computing the $L^2$-norm, we then get
	\begin{align*}
		\int_{\mathbb{R}^n_+}\Big|\int_{B_{\delta}(\hat{x}^\prime)\times(0,\delta)}J_\varepsilon(|A_\gamma(\hat{x},\hat{x})(\hat{x}-\bar{y})|)\partial_{x_n}u(\hat{x})(\hat{x}_n+\bar{y}_n)\big|\det \text{D}F_\gamma(\hat{x})\big|\;\text{d}y\Big|^2\text{d}\hat{x} 
		\leq K\|u\|_{H^2(\mathbb{R}^n_+)},
	\end{align*}
	where we used the same ideas as in Corollary \ref{cor3}, cf. \eqref{auxest2}.
	In the third term, it holds
	\begin{align*}
		&\int_{B_{\delta}(\hat{x}^\prime)\times(0,\delta)}J_\varepsilon(|A_\gamma(\hat{x},\hat{x})(\hat{x}-\bar{y})|)R_2(\hat{x},\bar{y})\big|\det \text{D}F_\gamma(\hat{x})\big|\;\text{d}y \\
		&\leq K \int_{\mathbb{R}^n_+}\int_{B_{\delta}(\hat{x}^\prime)\times(0,\delta)}\int_{0}^1\rho_\varepsilon(|A_\gamma(\hat{x},\hat{x})(\hat{x}-\bar{y})|)|D^2u(\bar{y}+t(\hat{x}-\bar{y}))|^2\:\text{d}t\text{d}y\text{d}\hat{x} \\
		&\leq K\int_{0}^1\int_{\mathbb{R}^{n-1}}\int_{\mathbb{R}^{n-1}}\|\rho_\varepsilon(|A_\gamma((\hat{x}^\prime,.),(\hat{x}^\prime,.))(\hat{x}^\prime-\bar{y}^\prime,.)|)\|_{L^1(\mathbb{R})}\|D^2u(\bar{y}^\prime+t(\hat{x}^\prime-\bar{y}^\prime),.)\|_{L^2(\mathbb{R}_+)}\:\text{d}\hat{x}^\prime\text{d}y^\prime\text{d}t \\
		&\leq K\|u\|_{H^2(\mathbb{R}^n_\gamma)}^2.
	\end{align*}
	Here, we used similar arguments as before, cf. \eqref{auxest1} and \eqref{auxest}.
	Altogether, we have
	\begin{align*}
		\|I^1_\varepsilon\|_{L^2(\mathbb{R}^n_\gamma)} \leq K\|c\|_{H^2(\mathbb{R}^n_\gamma)}.
	\end{align*}
	
	Ad $I^2_\varepsilon(x)$: Here, we have
	\begin{align*}
	|I_\varepsilon^2(x)| &\leq \Big|\int_{B_{\delta}(\hat{x}^\prime)\times(0,\delta)}J_\varepsilon(|A_\gamma(\hat{x},\bar{y})(\hat{x}-\bar{y})|)\big(\big|\det\text{D}F_\gamma(\bar{y})\big|-\big|\det\text{D}F_\gamma(\hat{x})\big|\big)\big(u(\hat{x})-u(\bar{y})\big)\;\text{d}y\Big| \\
	&+ \Big|\int_{B_{\delta}(\hat{x}^\prime)\times(0,\delta)}\Big(J_\varepsilon(|A_\gamma(\hat{x},\bar{y})(\hat{x}-\bar{y})|) - J_\varepsilon(|A_\gamma(\hat{x},\hat{x})(\hat{x}-\bar{y})|)\Big)\big|\det\text{D}F_\gamma(\hat{x})\big|\big(u(\hat{x})-u(\bar{y})\big)\;\text{d}y\Big|,
	\end{align*}
	where we used the same identity as in the proof of Lemma \ref{lemma4}. In the first term on the right-hand side, we observe that
	\begin{align*}
	\Big|\big(\big|\det\text{D}F_\gamma(\bar{y})\big|-\big|\det\text{D}F_\gamma(\hat{x})\big|\big)\big(u(\hat{x})-u(\bar{y})\big)\Big| \leq K\Big(\int_{0}^1|Du(\bar{y}+t(\hat{x}-\bar{y}))|\text{d}t\Big)|\hat{x}-\bar{y}|^2,
	\end{align*}
	which yields
	\begin{align*}
	&\int_{\mathbb{R}^n_+}\Big|\int_{B_{\delta}(\hat{x}^\prime)\times(0,\delta)}J_\varepsilon(|A_\gamma(\hat{x},\bar{y})(\hat{x}-\bar{y})|)\big(\big|\det\text{D}F_\gamma(\bar{y})\big|-\big|\det\text{D}F_\gamma(\hat{x})\big|\big)\big(u(\hat{x})-u(\bar{y})\big)\;\text{d}y\Big|^2\text{d}\hat{x} \\
	&\leq K\|u\|_{H^2(\mathbb{R}^n_+)}^2.
	\end{align*}
	Here, we concluded with the same arguments as in the proof before.
	In the second term, we again use the mean value theorem, cf. Lemma \ref{lemma4}, to obtain
	\begin{align*}
	&\int_{\mathbb{R}^n_+}\Big|\int_{B_{\delta}(\hat{x}^\prime)\times(0,\delta)}\Big(J_\varepsilon(|A_\gamma(\hat{x},\bar{y})(\hat{x}-\bar{y})|) - J_\varepsilon(|A_\gamma(\hat{x},\hat{x})(\hat{x}-\bar{y})|)\Big)\big|\det\text{D}F_\gamma(\hat{x})\big|\big(u(\hat{x})-u(\bar{y})\big)\;\text{d}y\Big|^2\text{d}\hat{x} \\
	&\leq K\int_{\mathbb{R}^n_+}\left(\int_{0}^1\int_{B_{\delta}(\hat{x}^\prime)\times(0,\delta)}\int_{0}^1\Big(\frac{1}{\varepsilon}\rho_\varepsilon^\prime(|z_t|)|\hat{x}-\bar{y}|+\rho_\varepsilon(|z_t|)\Big)\big|Du(\bar{y}+s(\hat{x}-\bar{y}))\big|\:\text{d}s\text{d}y\text{d}t\right)^2\text{d}\hat{x} \\ 
	&\leq K\|u\|_{H^2(\mathbb{R}^n_+)}^2, 
	\end{align*}
	where we defined 
	\begin{align*}
	z_t &:= M_t(\hat{x},\bar{y})(\hat{x}-\bar{y}), \\
	M_t(\hat{x},\bar{y}) &:= A_\gamma(\hat{x},\hat{x}) + t\big(A_\gamma(\hat{x},\bar{y}) - A_\gamma(\hat{x},\hat{x})\big)
	\end{align*}
	for all $\hat{x},\bar{y}\in\mathbb{R}^n$ and all $t\in[0,1]$. In fact, this estimate follows by the same methods as in the proof of Lemma \ref{lemma4}. Altogether, it follows that 
	\begin{align*}
		\|\mathcal{R}_\varepsilon\tilde{c}\|_{L^2(\mathbb{R}^n_\gamma)} \leq K\|c\|_{H^2(\mathbb{R}^n_\gamma)}
	\end{align*}
	for all $c\in C^{1,1}_0(\overline{\mathbb{R}^n_\gamma})$ with $\partial_{\mathbf{n}}c = 0$ on $\partial\mathbb{R}^n_\gamma$, and therefore
	\begin{align*}
		\Big\|\mathcal{L}_\varepsilon^{\mathbb{R}^n_\gamma} c + \Delta c\Big\|_{L^2(\mathbb{R}^n_\gamma)} \leq K\|c\|_{H^2(\mathbb{R}^n_\gamma)}
	\end{align*}	
	for all $c\in C^{1,1}_0(\overline{\mathbb{R}^n_\gamma})$ with $\partial_{\mathbf{n}}c = 0$ on $\partial\mathbb{R}^n_\gamma$. Finally, a denseness argument, cf. Lemma \ref{lemma4}, yields that this estimate also holds true for all $c\in H^2(\mathbb{R}^n_\gamma)$ with $\partial_{\mathbf{n}}c = 0$ on $\partial\mathbb{R}^n_\gamma$. In the last step, we employ the Banach Steinhaus Theorem. Then, the claim follows.
	\end{proof}
	\section{Main Result}\label{main}
	In this section, we want to state and prove the main result on the strong convergence of the nonlocal operator $\mathcal{L}_\varepsilon$ on bounded smooth domains $\Omega\subset\mathbb{R}^n$. 
	\newtheorem{theorem1}{Theorem}[section]
	\begin{theorem1}\label{theorem1}
		Let $\Omega\subset\mathbb{R}^n$, $n=2,3$, be a bounded domain with $C^2$-boundary. Then, for every $c\in H^2(\Omega)$ with $\partial_{\mathbf{n}}c = 0$ on $\partial\Omega$, it holds
		\begin{align}
			\Big\|\mathcal{L}_\varepsilon^\Omega c + \Delta c\Big\|_{L^2(\Omega)} \rightarrow 0\;\;\;\text{as }\varepsilon\searrow 0. 
		\end{align}
	\end{theorem1} 
	\begin{proof}
		Since $\partial\Omega$ is compact, there exist open sets $U_1,\ldots,U_N\subset\mathbb{R}^n$ and $\gamma_1,\ldots,\gamma_N\in C^2_0(\mathbb{R}^{n-1})$ such that, up to a rotation, $\Omega\cap U_j = \mathbb{R}^n_{\gamma_j}\cap U_j$ for all $j=1,\ldots,N$ and $\partial\Omega\subseteq \bigcup_{j=1}^N U_j$. Since also $\Omega\setminus\big(\bigcup_{j=1}^N U_j\big)$ is compact, there exists an open, bounded set $U_0\subset\mathbb{R}^n$ such that $\Omega\setminus\big(\bigcup_{j=1}^N U_j\big)\subset U_0$ and $\overline{U_0}\subset\Omega$. Then, we choose some $\gamma_0\in C^2_0(\mathbb{R}^{n-1})$ such that $\overline{U_0}\subset\mathbb{R}^n_{\gamma_0}$. Altogether, we have 
		\begin{align*}
		\overline{\Omega} \subseteq \bigcup_{j=0}^N U_j\;\;\;\text{and}\;\;\;\Omega\cap U_j = \mathbb{R}^n_{\gamma_j} \cap U_j\;\;\;\text{for all }j=0,\ldots,N.
		\end{align*}
		Next, we choose a partition of unity $\varphi_j$, $j=0,\ldots,N$, on $\overline{\Omega}$. Without loss of generality, we can assume that $\partial_{\mathbf{n}_j}{\varphi_j} = 0$ on $\partial\mathbb{R}^n_{\gamma_j}$ for all $j=0,\ldots,N$, where $\mathbf{n}_j$ denotes the outer unit normal to $\partial\mathbb{R}^n_{\gamma_j}$. Otherwise, $\partial\mathbb{R}^n_{\gamma_j}$ admits a tubular neighborhood $U_{a_j}$ of width $a_j>0$, cf. \cite[Section 2.3]{PrSi}. Then, the restriction of $\varphi_j$ to $\partial\mathbb{R}^n_{\gamma_j}$ can be extended to a function $\hat{\varphi_j}$ on $U_{a_j}$, which is constant in the direction of $\mathbf{n}_j$ by setting $\hat{\varphi_j}(x) := \varphi_j(\pi_{\partial\mathbb{R}^n_{\gamma_j}}(x))$ for $x\in U_{a_j}$, where $\pi_{\partial\mathbb{R}^n_{\gamma_j}}: U_{a_j} \rightarrow \partial\mathbb{R}^n_{\gamma_j}$ denotes the projection on $\partial\mathbb{R}^n_{\gamma_j}$.
		\newline
		Finally, we choose functions $\psi_j\in C^\infty_0(\mathbb{R}^n)$ such that $\text{supp }\psi_j\subseteq U_j$, $\psi_j\geq0$ and $\psi_j\equiv1$ in $\text{supp }\varphi_j$ for all $j=0,\ldots,N$.
		\newline
		\newline
		Now, let $c\in H^2(\Omega)$ with $\partial_{\mathbf{n}}c = 0$ on $\partial\Omega$ be arbitrary. Using the identity $c = \sum_{j=0}^N\varphi_jc$, we get
		\begin{align*}
		\mathcal{L}^\Omega_\varepsilon c + \Delta c &= \mathcal{L}^\Omega_\varepsilon\Big(\sum_{j=0}^N\varphi_jc\Big) + \Delta\Big(\sum_{j=0}^N\varphi_jc\Big) = \sum_{j=0}^N\Big(\mathcal{L}^\Omega_\varepsilon\big(\varphi_jc\big) + \Delta\big(\varphi_jc\big)\Big) \\
		&= \sum_{j=0}^N\Big(\psi_j\mathcal{L}^\Omega_\varepsilon\big(\varphi_jc\big) + (1-\psi_j)\mathcal{L}^\Omega_\varepsilon\big(\varphi_jc\big)+ \Delta\big(\varphi_jc\big)\Big) \\
		&=: \sum_{j=0}^N\Big(I_{\varepsilon,j}^1 + I_{\varepsilon,j}^2 + \Delta\big(\varphi_jc\big)\Big).
		\end{align*}
		In the next step, we analyze these terms separately. 
		
		Ad $I_{\varepsilon,j}^1$: For almost all $x\in\Omega$ and all $j=0,\ldots,N$ it holds:
		\begin{align*}
		\psi_j(x)\mathcal{L}^\Omega_\varepsilon\big(\varphi_jc\big)(x) &= \psi_j(x)\int_{\Omega}J_\varepsilon(|x-y|)\big(\varphi_j(x)c(x)- \varphi_j(y)c(y)\big)\;\text{d}y \\
		&= \psi_j(x)\int_{\mathbb{R}^n_{\gamma_j}}J_\varepsilon(|x-y|)\big(\varphi_j(x)c(x)- \varphi_j(y)c(y)\big)\;\text{d}y \\
		&\;\;\;+ \psi_j(x)\int_{\Omega\setminus U_j}J_\varepsilon(|x-y|)\varphi_j(x)c(x)\;\text{d}y - \psi_j(x)\int_{\mathbb{R}^n_{\gamma_j}\setminus U_j}J_\varepsilon(|x-y|)\varphi_j(x)c(x)\;\text{d}y \\
		&=: \hat{I}_{\varepsilon,j}^1 + \hat{I}_{\varepsilon,j}^2 - \hat{I}_{\varepsilon,j}^3.
		\end{align*}
		In $\hat{I}_{\varepsilon,j}^2$, we observe that $\delta_{2,j} := \text{dist}(\text{supp }\psi_j,\Omega\setminus U_j) > 0$ and therefore $|x-y|\geq\delta_{2,j}$. Hence,
		$\|\hat{I}_{\varepsilon,j}^2\|_{L^2(\Omega)}\rightarrow0$ as $\varepsilon\searrow0$ for all $j=0,\ldots,N$ by Lemma \ref{lemma2}.
		\newline
		In $\hat{I}_{\varepsilon,j}^3$, it holds $\delta_{3,j}:=\text{dist}(\text{supp }\psi_j,\mathbb{R}^n_{\gamma_j}\setminus U_j) > 0$. This implies $\|\hat{I}_{\varepsilon,j}^3\|_{L^2(\Omega)}\rightarrow0$ as $\varepsilon\searrow0$ for all $j=0,\ldots,N$ again by Lemma \ref{lemma2}.
		\newline
		In the next step, we consider $\hat{I}_{\varepsilon,j}^1 + \Delta\big(\varphi_jc\big)$. Since $\text{supp }\psi_j\subseteq U_j$, we have 
		\begin{align}\label{calc2}
		\hat{I}_{\varepsilon,j}^1 + \Delta\big(\varphi_jc\big) = \psi_j\mathcal{L}^{\mathbb{R}^n_{\gamma_j}}_\varepsilon\big(\varphi_jc\big) + \Delta\big(\varphi_jc\big) = \psi_j\Big(\mathcal{L}^{\mathbb{R}^n_{\gamma_j}}_\varepsilon\big(\varphi_jc\big) + \Delta\big(\varphi_jc\big)\Big).
		\end{align}
		Due to the properties of $\varphi_j$, it holds $\varphi_jc\in H^2(\mathbb{R}^n_{\gamma_j})$ and $\partial_\mathbf{n}(\varphi_jc)=0$ on $\partial\mathbb{R}^n_{\gamma_j}$. Therefore, \eqref{calc2} vanishes as $\varepsilon\searrow0$ by Corollary \ref{cor4}. 
		
		Ad $I_{\varepsilon,j}^2$: Due to our choice of the functions $\psi_j$, it holds $\text{supp}(1-\psi_j)\cap\text{supp }\varphi_j=\emptyset$. Thus, we have $\delta_j := \text{dist}(\text{supp}(1-\psi_j),\text{supp }\varphi_j) >0$. Therefore, $\|I_{\varepsilon,j}^2\|_{L^2(\Omega)} \rightarrow 0$ as $\varepsilon\searrow0$ for all $j=0,\ldots,N$ by Lemma \ref{lemma2}. This concludes the proof.
	\end{proof}
	\newtheorem{cor6}[theorem1]{Corollary}	
	\begin{cor6}\label{cor6}
		In addition, if $c\in H^3(\Omega)$ and $\Omega$ is of class $C^3$, it even holds
		\begin{align*}
			\Big\|\mathcal{L}_\varepsilon^\Omega c + \Delta c\Big\|_{L^2(\Omega)} \leq K\sqrt{\varepsilon}\|c\|_{H^3(\Omega)}.
		\end{align*}
	\end{cor6}
	\begin{proof}
		This can be proven analogously as Theorem \ref{theorem1} using Lemma \ref{lemma4} instead of Corollary \ref{cor4}. 
	\end{proof}
\section{Nonlocal-to-Local Convergence of the Cahn-Hilliard Equation}\label{nonloctoloc}
	\newtheorem{cor1}{Theorem}[section]
	\begin{cor1}\label{cor1}
		Let the initial data $c_{0,\varepsilon}\in L^2(\Omega)$ satisfy $c_{0,\varepsilon} \rightarrow c_0$ in $L^2(\Omega)$ at rate $\mathcal{O}(\sqrt{\varepsilon})$ as $\varepsilon\searrow 0$ for $c_0\in H^1(\Omega)$. 
		Let the weak solution 
		$c\in L^\infty(0,T;H^1(\Omega))\cap L^2(0,T;H^2(\Omega))$ with $f^\prime(c) \in L^2(\Omega_T)$
		of the local Cahn-Hilliard equation satisfy 
		$c\in L^2(0,T;H^3(\Omega))$.
		Then, the weak solution $c_\varepsilon$ of the nonlocal Cahn-Hilliard equation \eqref{NLCH1}-\eqref{NLCH2} converges strongly to the strong solution of the local Cahn-Hilliard equation \eqref{CH1}-\eqref{CH2} in $L^\infty(0,T;H^{-1}(\Omega))\cap L^2(\Omega_T)$ at rate $\mathcal{O}(\sqrt{\varepsilon})$ as $\varepsilon\searrow 0$. 
	\end{cor1}
	\begin{proof}
		Let us define the functions $u:= c_\varepsilon - c$ and $w:= \mu_\varepsilon - \mu$. Then, $(u,w)$ solves the system
		\begin{align}
			\partial_tu &= \Delta w, \label{equ3} \\
			w &= \mathcal{L}_\varepsilon^\Omega c_\varepsilon + \Delta c + f^\prime(c_\varepsilon) - f^\prime(c) .\label{equ4}
		\end{align}
		In the next step, we test \eqref{equ3} by $(-\Delta_N)^{-1}u$. 
		Here, $v:=(-\Delta_N)^{-1}u \in H^1(\Omega)\cap L^2_{(0)}(\Omega)$ denotes the unique solution of 
		\begin{align*}
			\int_{\Omega}\nabla v\cdot\nabla\varphi\:\text{d}x = \int_{\Omega}u\varphi\:\text{d}x
		\end{align*}
		for all $\varphi\in H^1(\Omega)\cap L^2_{(0)}(\Omega)$.
		This yields
		\begin{align}\label{equ5}
			\frac{\text{d}}{\text{d}t}\frac{1}{2}\|u\|_{H^{-1}(\Omega)}^2 = \int_{\Omega}\Delta w(-\Delta_N)^{-1}u\:\text{d}x = - \int_{\Omega}wu\:\text{d}x.
		\end{align}
		Here, we also used the following calculation
		\begin{align*}
		\int_{\Omega}\partial_tu(-\Delta_N)^{-1}u\:\text{d}x = \frac{\text{d}}{\text{d}t}\int_{\Omega}\frac{1}{2}\big|(-\Delta_N)^{-1/2}u\big|^2\:\text{d}x = \frac{\text{d}}{\text{d}t}\frac{1}{2}\|u\|^2_{H^{-1}(\Omega)}.
		\end{align*}
		Furthermore, we multiply \eqref{equ4} by $u$ and integrate over $\Omega$. Hence,
		\begin{align*}
			\int_{\Omega}wu\:\text{d}x = \int_{\Omega}\mathcal{L}^{\Omega}_\varepsilon c_\varepsilon u\:\text{d}x + \int_{\Omega}\Delta cu\:\text{d}x + \int_{\Omega}\big(f^\prime(c_\varepsilon) - f^\prime(c)\big)u\:\text{d}x.
		\end{align*}
		Using the assumptions on $f$, we can estimate the last integral on the right-hand side from below by
		\begin{align*}
		\int_{\Omega}\big(f^\prime(c_\varepsilon) - f^\prime(c)\big)u\:\text{d}x \geq -\int_{\Omega}\alpha(c_\varepsilon - c)u\:\text{d}x = -\int_{\Omega}\alpha|u|^2\:\text{d}x.
		\end{align*}
		Therefore, we have
		\begin{align}
		\int_{\Omega}wu\:\text{d}x &\geq \int_{\Omega}\mathcal{L}^{\Omega}_\varepsilon uu\:\text{d}x + \int_{\Omega}\big(\mathcal{L}^{\Omega}_\varepsilon c +\Delta c\big)u\:\text{d}x -\alpha\int_{\Omega}|u|^2\:\text{d}x \nonumber\\
		&= \mathcal{E}_\varepsilon(u) + \int_{\Omega}\big(\mathcal{L}^{\Omega}_\varepsilon c +\Delta c\big)u\:\text{d}x -\alpha\int_{\Omega}|u|^2\:\text{d}x, \label{calc3}
		\end{align}
		where we used 
		\begin{align}\label{E}
		\int_{\Omega}\mathcal{L}^{\Omega}_\varepsilon uu\:\text{d}x = \frac{1}{4}\int_{\Omega}\int_{\Omega}J_\varepsilon(x-y)\big|u(x) - u(y)\big|^2\:\text{d}x\:\text{d}y = \mathcal{E}_\varepsilon(u)
		\end{align}
		in the last step. Now, combining \eqref{equ5} and \eqref{calc3} yields
		\begin{align*}
		\frac{\text{d}}{\text{d}t}\frac{1}{2}\|u\|^2_{H^{-1}(\Omega)} + \frac{1}{2}\|u\|^2_{L^2(\Omega)} + \mathcal{E}_\varepsilon(u) &\leq \Big(\alpha+\frac{1}{2}\Big)\|u\|^2_{L^2(\Omega)} - \int_{\Omega}\big(\mathcal{L}^{\Omega}_\varepsilon c+\Delta c\big)u\:\text{d}x \nonumber\\
		&\leq \Big(\alpha + 1\Big)\|u\|^2_{L^2(\Omega)} + \frac{1}{2}\Big\|\mathcal{L}^{\Omega}_\varepsilon c + \Delta c\Big\|^2_{L^2(\Omega)}.
		\end{align*} 
		Employing inequality \eqref{ineqDav} from Lemma \ref{lemma0} with $\delta = \frac{1}{2(\alpha+1)}$, we then obtain
		\begin{align}\label{equ6}
		\frac{\text{d}}{\text{d}t}\frac{1}{2}\|u\|^2_{H^{-1}(\Omega)} + \frac{1}{2}\|u\|^2_{L^2(\Omega)} + \frac{1}{2}\mathcal{E}_\varepsilon(u) \leq C\|u\|^2_{H^{-1}(\Omega)} + C\Big\|\mathcal{L}^{\Omega}_\varepsilon c + \Delta c\Big\|^2_{L^2(\Omega)}.
		\end{align}
		Finally, Gronwall`s inequality yields
		\begin{align*}
			\|u(t)\|_{H^{-1}(\Omega)}^2 + \frac{1}{2}\|u\|_{L^2(\Omega_t)}^2 + \int_{0}^t\mathcal{E}_\varepsilon(u)(\tau)\:\text{d}\tau \leq K\sqrt{\varepsilon}\Big(1+\int_{0}^t\|c(\tau)\|_{H^3(\Omega)}^2\:\text{d}\tau\Big)\exp\Big(\int_{0}^tK\:\text{d}\tau\Big)
		\end{align*}
		for almost all $t\in(0,T)$, where we used Corollary \ref{cor6} and the assumptions on the initial data on the right-hand side. Using \eqref{nonlockonv}, we then conclude the proof.
	\end{proof}
	\newtheorem{cor2}[cor1]{Corollary}
	\begin{cor2}
		Under the assumptions of Corollary \ref{cor1}, it even holds
		\begin{align*}
			c_\varepsilon \rightarrow c\;\;\;\text{in }L^\infty(0,T;H^s(\Omega))\;\;\;\text{ as }\varepsilon\searrow 0
		\end{align*} 
		for all $s\in(-1,0)$ and furthermore,
		\begin{align*}
			c_\varepsilon \rightarrow c\;\;\;\text{in }L^2(0,T;L^p(\Omega))\;\;\;\text{ as }\varepsilon\searrow 0
		\end{align*}
		for all $p\in[2,6)$.
		\begin{proof}
			In Theorem \ref{cor1}, we have already shown that
			 \begin{align*}
			 c_\varepsilon \rightarrow c\;\;\;\text{in }L^\infty(0,T;H^{-1}(\Omega))
			 \end{align*} 
			 as $\varepsilon\searrow 0$. Since the sequence $(c_\varepsilon)_{\varepsilon>0}\subset L^\infty(0,T;L^2(\Omega))$ is bounded, cf. \cite[Theorem 2.3]{Davoli2}, we get
			 \begin{align*}
			 	\|c_\varepsilon-c\|_{L^\infty(0,T;H^s(\Omega))} \leq K\|c_\varepsilon-c\|_{L^\infty(0,T;L^2(\Omega))}^\theta\;\|c_\varepsilon-c\|_{L^\infty(0,T;H^{-1}(\Omega))}^{1-\theta}
			 \end{align*}
			 for $s=\theta-1$, $\theta\in(0,1)$ and thus the first assertion follows. 
			 
			 In the second assertion, we use that $(c_\varepsilon)_{\varepsilon>0}\subset L^2(0,T;H^1(\Omega)) \hookrightarrow L^2(0,T;L^6(\Omega))$ is bounded, cf. \cite[Theorem 2.2]{Davoli2}. Then, an interpolation yields
			 \begin{align*}
			 	\|c_\varepsilon-c\|_{L^2(0,T;L^p(\Omega))} \leq K\|c_\varepsilon-c\|_{L^2(\Omega_T)}^\theta\;\|c_\varepsilon-c\|_{L^2(0,T;L^6(\Omega))}^{1-\theta}
			 \end{align*}
			 for all $p=\frac{6}{1+2\theta}$, where $\theta\in(0,1]$. This finishes the proof. 
		\end{proof}
	\end{cor2}
	\newtheorem{rem2}[cor1]{Remark}
	\begin{rem2}
		\upshape
		 Lastly, we want to mention that our methods from Section \ref{nonloctoloc} can also be used to prove nonlocal-to-local convergence for the Allen-Cahn equation. The Allen-Cahn equation is given by
		 \begin{align}\label{AC}
		 	\partial_tc = \Delta c - f^\prime(c)\;\;\;\text{in }\Omega_T
		 \end{align}
		 together with boundary and initial conditions
		 \begin{align}
		 	\partial_{\mathbf{n}}c &= 0\;\;\;\;\text{on }\partial\Omega\times(0,T), \\
		 	c|_{t=0} &= c_0\;\;\;\text{in }\Omega\label{ACIC}.
		 \end{align}
		 The nonlocal version is given by
		 \begin{align}\label{NLAC}
		 	\partial_tc = -\mathcal{L}_\varepsilon c - f^\prime(c)\;\;\;\text{in }\Omega_T
		 \end{align}
		 with initial condition
		 \begin{align}\label{NLACIC}
		 	c|_{t=0} &= c_{0}\;\;\;\text{in }\Omega.
		 \end{align}
		 Observe that we do not need to impose any boundary condition for the nonlocal Allen-Cahn equation. Here, we use the same notation and the same prerequisites as before, cf. Section \ref{intro}. Then, we can prove the following assertion.
	\end{rem2}
	\newtheorem{cor5}[cor1]{Theorem}
	\begin{cor5}
			Let the initial data $c_{0,\varepsilon}\in L^2(\Omega)$ satisfy $c_{0,\varepsilon} \rightarrow c_0$ at rate $\mathcal{O}(\sqrt{\varepsilon})$ in $L^2(\Omega)$ as $\varepsilon\searrow 0$ for $c_0\in H^1(\Omega)$. 
			Let the weak solution
			$c\in L^\infty(0,T;H^1(\Omega))\cap L^2(0,T;H^2(\Omega))$ with $f^\prime(c)\in L^2(\Omega_T)$
			of the local Allen-Cahn equation satisfy $c\in L^2(0,T;H^3(\Omega))$.
			Then, the weak solution $c_\varepsilon$ of the nonlocal Allen-Cahn equation \eqref{NLAC}-\eqref{NLACIC} converges strongly to the strong solution of the local Allen-Cahn equation \eqref{AC}-\eqref{ACIC} in $L^\infty(0,T;L^2(\Omega))$ at rate $\mathcal{O}(\sqrt{\varepsilon})$ as $\varepsilon\searrow0$. 
	\end{cor5}
	\begin{proof}
		First of all, we define $u:= c_\varepsilon - c$. Then, $u$ is a solution of 
		\begin{align}\label{ACu}
			\partial_tu = - \mathcal{L}^\Omega_\varepsilon c_\varepsilon -f^\prime(c_\varepsilon) - \Delta c + f^\prime(c).
		\end{align}
		Testing \eqref{ACu} with $u$, yields
		\begin{align*}
			\int_{\Omega}\partial_tuu\;\text{d}x &= -\int_{\Omega}\big(\mathcal{L}^\Omega_\varepsilon c_\varepsilon + \Delta c\big)u\;\text{d}x - \int_{\Omega}\big(f^\prime(c_\varepsilon) - f^\prime(c)\big)u\;\text{d}x \\
			&= -\int_{\Omega}\big(\mathcal{L}^\Omega_\varepsilon c + \Delta c\big)u\;\text{d}x - \int_{\Omega}\mathcal{L}^\Omega_\varepsilon uu\;\text{d}x-\int_{\Omega}\big(f^\prime(c_\varepsilon) - f^\prime(c)\big)u\;\text{d}x,
		\end{align*}
		where the left-hand side also satisfies 
		\begin{align*}
			\int_{\Omega}\partial_tuu\;\text{d}x = \frac{\text{d}}{\text{d}t}\frac{1}{2}\|u\|_{L^2(\Omega)}^2.
		\end{align*}
		Employing the properties of $f$ and \eqref{E}, we then end up with
		\begin{align*}
			\frac{\text{d}}{\text{d}t}\frac{1}{2}\|u\|_{L^2(\Omega)}^2 + \mathcal{E}_\varepsilon(u) 
			\leq \Big(\alpha + \frac{1}{2}\Big)\|u\|^2_{L^2(\Omega)} + \frac{1}{2}\Big\|\mathcal{L}^{\Omega}_\varepsilon c + \Delta c\Big\|^2_{L^2(\Omega)}. 
		\end{align*}
		In the last step, we use inequality \eqref{ineqDav} from Lemma \ref{lemma0} with $\delta = \frac{1}{2(\alpha+\frac{1}{2})}$, Gronwall`s inequality as well as Theorem \ref{theorem1} and \eqref{nonlockonv}. Then, the claim follows.
	\end{proof}
	\section*{Acknowledgments}
	The second author was partially supported by the Graduiertenkolleg 2339 \textit{IntComSin} of the Deutsche Forschungsgemeinschaft  (DFG, German Research Foundation) -- Project-ID 321821685. The support is gratefully acknowledged.
	%\newpage
	\renewcommand{\refname}{References}
	

\begin{thebibliography}{9}
		\bibitem{AbelsTerasawa1}
		H. Abels, Y. Terasawa. \textit{Convergence of a Nonlocal to a Local Diffuse Interface Model for Two-Phase Flow with Unmatched Densities}. Discrete Contin. Dyn. Syst. Ser. S15(2022), no.8, 1871–1881.
		\bibitem{AbelsTerasawa}
		H. Abels, Y. Terasawa. \textit{On Stokes operators with variable viscosity in bounded and unbounded domains.} Math. Ann. 344, 381–429 (2009). https://doi.org/10.1007/s00208-008-0311-7.
		\bibitem{Bourgain}
		J. Bourgain, H. Brezis, P. Mironescu. \textit{Another look at Sobolev spaces.} In Optimal control and
		partial differential equations, pages 439–455. IOS, Amsterdam, 2001.
		\bibitem{Burbovska}
		O. Burkovska, M. Gunzburger. \textit{On a nonlocal Cahn-Hilliard model permitting sharp interfaces}. Math. Models Methods Appl. Sci.31(2021), no.9, 1749–1786.
		\bibitem{Cahn}
		J.W. Cahn, J.E. Hilliard. \textit{Free energy of a nonuniform system I. Interfacial free energy.} J. Chem. Phys.,
		2:205–245, 1958.
		\bibitem{Crismale}
		V. Crismale, L. De Luca, A. Kubin, A. Ninno, M. Ponsiglione. \textit{The variational approach to $s$-fractional heat flows and the limit cases $s\rightarrow0^+$ and $s\rightarrow1^-$}. J. Funct. Anal.284(2023), no.8, Paper No. 109851, 38 pp.
		\bibitem{Davoli1}
		E. Davoli, H. Ranetbauer, L. Scarpa, L. Trussardi. \textit{Degenerate nonlocal Cahn-Hilliard equations: Well-posedness, regularity and local asymptotics}. Ann. Inst. H. Poincaré Anal. Non Linéaire 37 (2020), no. 3, pp. 627–651.
		\bibitem{Davoli2}
		E. Davoli, L. Scarpa, L. Trussardi. \textit{Local asymptotics for nonlocal convective
			Cahn-Hilliard equations with $W^{1,1}$-kernel and singular potential}. J. Differential Equations,
		289 (2021), pp. 35–58.
		\bibitem{Davoli3}
		E. Davoli, L. Scarpa, L. Trussardi. \textit{Nonlocal-to-Local Convergence of Cahn–Hilliard Equations: Neumann Boundary Conditions and Viscosity Terms.} Arch Rational Mech Anal 239, 117–149 (2021). https://doi.org/10.1007/s00205-020-01573-9.
		\bibitem{Elbar}
		C. Elbar, J. Skrzeczkowski. \textit{Degenerate Cahn-Hilliard equation: From nonlocal to local}. J. Differential Equations, 364:576–611, 2023.
		\bibitem{Kassmann}
		G.F. Foghem Gounoue, M. Kassmann, P. Voigt. \textit{Mosco convergence of nonlocal to local quadratic forms}. Nonlinear Anal.193(2020), 111504, 22 pp.
		\bibitem{Gal}
		C.G. Gal, A. Giorgini, M. Grasselli. \textit{The nonlocal Cahn-Hilliard equation with singular potential:
			well-posedness, regularity and strict separation property.} J. Differential Equations, 263(9):5253–5297,
		2017.
		\bibitem{Giacomin}
		G. Giacomin, J. Lebowitz. \textit{Phase segregation dynamics in particle systems with long range interactions. I. Macroscopic limits.} J. Stat. Phys, 87(1):37–61, 1997.
		\bibitem{Knopf}
		P. Knopf, A. Signori. \textit{On the nonlocal Cahn–Hilliard equation with nonlocal dynamic boundary condition
			and boundary penalization.} J. Differential Equations, 280(4):236–291, 2021. 
		\bibitem{Kurima}
		S. Kurima. \textit{Nonlocal to local convergence of singular phase field systems of conserved type}. Adv. Math. Sci. Appl.31(2022), no.2, 481–500.
		\bibitem{Liang}
		Z. Liang. \textit{On the existence of weak solutions to non-local Cahn-Hilliard/Navier-Stokes equations and its local asymptotics}. Commun. Math. Sci.18(2020), no.8, 2121–2147.
		\bibitem{Melchoinna}
		S. Melchionna, H. Ranetbauer, L. Scarpa, and L. Trussardi. \textit{From nonlocal to
			local Cahn-Hilliard equation}. Adv. Math. Sci. Appl., 28 (2019), p. 197–211.
		\bibitem{Ponce1}
		A.C. Ponce. \textit{An estimate in the spirit of Poincaré’s inequality.} Journal of the European Mathematical
		Society, 6(1):1–15, 2004.
		\bibitem{Ponce2}
		A.C. Ponce. \textit{A new approach to Sobolev spaces and connections to $\Gamma$-convergence.} Calc. Var. Partial
		Differential Equations, 19(3):229–255, 2004.
		\bibitem{PrSi}
		J. Prüss, G. Simonett. \textit{Moving Interfaces and Quasilinear Parabolic Evolution Equations}. Volume 105 of Monographs in Mathematics (Birkhäuser/Springer, Cham, 2016).
		\bibitem{Schumacher}
		K. Schumacher. \textit{A chart preserving the normal vector and extensions of normal derivatives in weighted function spaces.} Czech Math J 59, 637–648 (2009). https://doi.org/10.1007/s10587-009-0057-8.
		\bibitem{Vasquez}
		J.L. Vázquez. \textit{Asymptotic behaviour for the fractional heat equation in the Euclidean space}. Complex Var. Elliptic Equ.63(2018), no.7-8, 1216–1231.
	\end{thebibliography}
\end{document}